\documentclass[letterpaper, 10 pt, article]{IEEEtran}
\onecolumn
\usepackage{amsmath, amssymb,amsthm,cite}
\usepackage[utf8x]{inputenc}
\usepackage{graphicx}
\usepackage{psfrag}
\usepackage{bbm}
\usepackage{mathtools}
\usepackage{color}
\usepackage{tcolorbox}
\usepackage{arydshln}
\usepackage{subcaption}

\DeclareMathOperator*{\nn}{\nonumber}

\DeclareMathOperator*{\cost}{cost}

\allowdisplaybreaks
\newtheorem{lemma}{Lemma}
\newtheorem{theorem}{Theorem}

\newtheorem{remark}{Remark}
\newtheorem{problem}{Problem}
\theoremstyle{definition}

\def\rr{\textcolor{black}}

\renewcommand{\v}[1]{\mathbf{#1}}

\DeclareMathOperator{\regret}{\operatorname{Regret}}
\DeclarePairedDelimiterX{\norm}[1]{\lVert}{\rVert}{#1}

\makeatletter
\def\blfootnote{\gdef\@thefnmark{}\@footnotetext}
\makeatother

\title{\textcolor{black}{Regret-Optimal LQR Control}}
\begin{document}
\author{Oron Sabag, Sahin Lale, Gautam Goel, Babak Hassibi}
\maketitle
\begin{abstract}
We consider the infinite-horizon LQR control problem. \rr{Motivated by competitive analysis in online learning,} as a criterion for controller design we introduce the dynamic regret, defined as the difference between the LQR cost of a causal controller (that has only access to past disturbances) and the LQR cost of the \emph{unique} clairvoyant one (that has also access to future disturbances) that is known to dominate all other controllers. The regret itself is a function of the disturbances, and we propose to find a causal controller that minimizes the worst-case regret over all bounded energy disturbances. The resulting controller has the interpretation of guaranteeing the smallest regret compared to the best non-causal controller that can see the future. We derive explicit formulas for the optimal regret and for the regret-optimal controller for the state-space setting. These explicit solutions are obtained by showing that the regret-optimal control problem can be reduced to a Nehari extension problem that can be solved explicitly. The regret-optimal controller \rr{is shown to be linear} and can be expressed as the sum of the classical $\mathcal H_2$ state-feedback law and an $n$-th order controller ($n$ is the state dimension), and its construction simply requires a solution to the standard LQR Riccati equation and two Lyapunov equations. Simulations over a range of plants demonstrate that the regret-optimal controller interpolates nicely between the $\mathcal H_2$ and the $\mathcal H_\infty$ optimal controllers, and generally has $\mathcal H_2$ and $\mathcal H_\infty$ costs that are simultaneously close to their optimal values. The regret-optimal controller thus presents itself as a viable option for control systems design.
\end{abstract}
\section{Introduction}
\blfootnote{O. Sabag was with the Department of Electrical Engineering at California Institute of Technology. He is now with the Rachel and Selim Benin School of Computer Science and Engineering, Hebrew University of Jerusalem, (email: oron.sabag@mail.huji.ac.il). G. Goel is with the Simons Institute, UC Berkeley (e-mail:ggoel@berkeley.edu) S. Lale and B. Hassibi are with the Department of Electrical Engineering at California Institute of Technology (e-mails:
 \{alale,hassibi\}@caltech.edu).}
\blfootnote{A preliminary version of this paper has been published in the ACC 2021~\cite{sabagFIACC}.}
In this paper, we consider control through the lens of \textit{regret minimization}. While the literature on control is vast, control theorists have largely studied control in two distinct settings. In one setting, we assume that the disturbances are generated by random processes whose statistics we know (in the Gaussian case this is LQG control, in the iid case with linear controllers it is $\mathcal H_2$ control), and the goal is to design a control policy which minimizes the expected control cost. In the other setting, robust control, there are no distributional assumptions about the disturbance and we seek to minimize the worst-case gain across all bounded disturbances (for bounded energy or power this is $\mathcal H_\infty$ control \cite{29425,zames1981feedback}, for bounded amplitude it is $\ell_1$ control \cite{1104603}). \rr{As a result, $\mathcal H_\infty$ are overly conservative since they safeguard against the worst-case, and $\mathcal H_2$ control are sensitive to modeling inaccuracies \cite{doyle1978guaranteed}. In this paper, we propose a new approach to deal with future uncertainty, which is based on a competitive criterion called the regret.}

In the regret framework, instead of trying to design controllers that achieve optimal performance relative to a certain class of disturbances, we propose to track the behavior of the benchmark non-causal controller. In particular, our criterion is the regret that measures the largest deviation from the non-causal controller, a regret-optimal controller is a one that aims to maintain \rr{a balanced performance across all disturbances}, regardless of whether the disturbances are stochastic, adversarial, etc. This stands in stark contrast to $\mathcal H_2$ and $\mathcal H_{\infty}$ control, which generally yield controllers that perform well in the environments they are designed for but whose performance can degrade badly when placed in different environments. This is made transparent in our regret problem formulation in \eqref{eq:comparison_robust}.

Regret minimization in control problems has attracted much recent interest in \rr{online learning} (see, e.g., \cite{abbasi2011regret,agarwal19c,foster20b,cohen19b,lale2020explore,dean2018regret,lale2020logarithmic,agarwal2019logarithmic,li2019online,hazan2019nonstochastic,goel2020power} and the references therein). In this paper, we study the so-called full-information control problem, where at every time instant the causal controller has access to past and current disturbances. Most papers in this area try to design causal controllers that compete with the best static linear state feedback controller selected in hindsight; in other words, they compete with the controller which in every round sets the control action $u_t$ to be $K x_t$ where $x_t$ is the state and $K$ is a {\em fixed} matrix selected with full clairvoyant knowledge of the disturbances. We believe this choice of non-causal controller to be rather unnatural:
\begin{center}
    \textit{Why restrict to static linear state feedback?}
\end{center}

It is not clear when or whether such a non-causal policy outperforms a causal controller that is allowed to be an arbitrary causal function of the states. We therefore would like to contend that it is more natural to design controllers which compete with the best sequence of control actions selected in hindsight, not just those generated by static linear state feedback. In other words, we seek to design controllers that compete with the optimal sequence of control actions $u_1^* \ldots u_{T-1}^*$, without imposing the restriction that the benchmark sequence satisfy $u_t^* = -K x_t$ for some fixed matrix $K$.

We utilize the regret metric to design a causal controller by comparing its performance with the unique optimal non-causal controller. The optimal regret is defined as the largest deviation in the LQR costs of the causal and the non-causal controller among all bounded energy disturbances. \rr{The motivation behind the regret definition is to construct a causal controller that aims to mimic the behavior of the non-causal controller by minimizing the regret distance.} At the operator level, we show that the regret problem can be reduced to the classical Nehari problem \cite{nehari1957bounded}. For the state-space setting, we derive the optimal regret as a simple formula and provide an explicit regret-optimal controller that is given by an explicit state-space realization. The resulting controller inherits the finite-dimensional state-space of the underlying system and its implementation requires the computation of the standard LQR Riccati equation and two additional Lyapunov equations.  

The rest of the paper is organized as follows. In Section \ref{sec:setting}, we present the problem formulation. Section \ref{sec:main} includes our main results, and Section \ref{sec:numerical} contains numerical simulations. Section \ref{sec:proof} includes the proofs, and the paper is concluded in Section \ref{sec:conclusion}. %\bb{The appendix includes some extended proofs will be referred to \cite{extended}.}

\section{The Setting and Problem Formulation}\label{sec:setting}
In this section, we introduce the setting and the regret-optimal control problem.
% \vspace{-0.5em}
\subsection{Notation}
\rr{The Euclidean norm of a vector $x$ is $\|x\|_2$. The operator norm of $A$ is $\|A\|$, and $\|A\|_F$ denotes its Frobenius norm. We use $A^\ast$ to denote the conjugate transpose of $A$. If $A$ has only real eigenvalues, its largest eigenvalue is denoted by $ \lambda_{\text{max}}(A)$. The strictly causal part (its strictly lower triangular part) and the anticausal part (its upper triangular part) of an operator $A$ are denoted by $\{A\}_+$ and $\{A\}_-$, respectively. An operator $A$ is said to be strictly causal if $\{A\}_-=0$. We use $I$ to denote the identity matrix when the dimensions are clear.}
% \vspace{-0.3em}
\subsection{The state-space setting}\label{subsec:ss}
We consider time-invariant dynamical systems given by
\begin{align}\label{eq:dynamical_sys}
    x_{t+1} = A x_t + B_u u_t + B_w w_t,
\end{align}
where $x_t \in \mathbb{R}^n$ is the state variable, $u_t \in \mathbb{R}^m$ is the control variable which we can dynamically adjust to influence the evolution of the system, and $w_t \in \mathbb{R}^p$ is the disturbance. It is also assumed that the pair $(A,B_u)$ is stabilizable.

A policy $\pi$ is defined as a mapping of disturbance sequences $w = \{w_t\}$ to control sequences $u = \{u_t\}$\footnote{\rr{We later show that the regret-optimal controller can be computed with the states only, but it is convenient to define policy as a function of the disturbance.}}. We focus on the doubly infinite-horizon regime where, for a fixed policy $\pi$, the linear–quadratic regulator (LQR) cost is given by
\begin{align}\label{eq:cost_def_DoublyInf}
  \cost(\pi;w) = \sum_{t=-\infty}^\infty \left(x_t^\ast Qx_t+u_t^\ast Ru_t\right),
\end{align}
for $Q, R \succ\! 0$. For \eqref{eq:cost_def_DoublyInf} to be finite (and meaningful) disturbances are assumed to have bounded energy, i.e., $w \!=\! \{w_t\}\!\in\!\ell_2$.\looseness=-1

In $\mathcal H_2$ and $\mathcal H_\infty$ control, the objective is to design a controller which minimizes the LQR cost under different assumptions on the disturbances sequence $w$. \rr{Our approach is different since we optimize a competitive criterion that compares the LQR cost of the causal controller (to be designed) with the cost of an unrealizable controller that serves as a benchmark. Specifically, we use the regret criterion defined as the difference between the LQR costs of a causal controller and the optimal non-causal controller.} Formally, define the set of non-causal policies as $\Pi^{\text{N.C.}}$ that map the disturbance sequence $w$ to the sequence $u$. The set $\Pi^{\text{N.C.}}$ is not restricted to be linear, time-invariant etc. The regret of a policy $\pi$ is then defined as
\begin{align}\label{eq:def_regret}
  \regret(\pi) =\!\! \sup_{\|w\|_2\le 1}\!\! \left(\cost(\pi;w) - \inf_{\pi'\in \Pi^{\text{N.C.}}} \cost(\pi';w)\right).
\end{align}
The regret criterion is now clear; the performance of a controller, defined by a policy $\pi$, is measured with respect to the performance of the best non-causal controller. An important feature of \eqref{eq:def_regret} is that the disturbance $w$ plays a role in both costs and, therefore, the comparison between the policies is meaningful. \rr{The motivation to minimize regret is to design a controller whose performance aims to mimic the performance of the non-causal controller as much as possible. For disturbances for which (even) the non-causal controller has a large control cost, the designed controller may have a higher cost. However, if for a certain disturbance a lower cost is attainable for the non-causal controller, then the designed controller should have a low cost as well. The regret captures this behavior and results in a controller that is competitive with respect to the benchmark non-causal controller.}

In \eqref{eq:def_regret}, we assume that the disturbance $w$ has bounded energy, i.e., $w$ is an $\ell_2$ sequence. An alternative formulation is to define the cost as $\lim_{T\rightarrow\infty}\frac{1}{T}\cost_T(\pi;w)$ where $\cost_T$ is cumulative cost and to assume that the sequence $w$ has bounded power, i.e.,~$\lim_{T\rightarrow\infty}\frac{1}{T}{\sum_{t=0}^{T-1}}w_t^*w_t<\infty$. In $\mathcal H_\infty$ theory both formulations lead to the same optimal controller.

% \bb{cite a paper/book for this claim. Should the sum begin with $t=0$ or $t=-\infty$?}.

We consider strictly-causal policies, that is, the strictly causal controller chooses $u_t$ when it has access to $\{w_i\}_{i< t}$. The set of strictly causal policies is denoted by $\Pi^{S.C.}$. The regret-optimal control problem can be summarized as follows.
\begin{problem}[Regret-Optimal Control]\label{prob:regret} Find a strictly-causal policy $\pi$ that solves the optimization problem
    \begin{align}
    {\regret} = \inf_{\pi\in\Pi^{\text{S.C.}}}\regret(\pi).
    \end{align}
\end{problem}
In Section \ref{sec:main}, Problem \ref{prob:regret} is completely solved in terms of explicit formulae for the optimal regret and construction of a regret-optimal controller. A solution to the regret-optimal control problem in the causal scenario can be found in \cite{sabagFIACC}.

% The optimal regret$^\ast$ in \eqref{eq:def_opt_regret} characterizes the inherent loss in terms of the cost of the best online controller compared to its best noncausal counterpart. Unlike most earlier papers that only give order-wise results on the optimal regret, we are able to solve Problem \ref{prob:regret} {\em exactly}. In other words, we provide an explicit formula to compute regret$^\ast$ and the optimal online policy in both the causal strictly causal scenarios.

\subsection{Regret-optimal control in operator form}\label{subsec:op}
In this section, we present the regret problem from an operator theory perspective. This leads to clean exposition and insightful comparison with $\mathcal H_\infty$ control. The state-space in \eqref{eq:dynamical_sys} is a special case of this formulation. Consider a linear system
\begin{align}\label{eq:operator_sys}
    s = Fv+Gw,
\end{align}
where $F$ and $G$ are causal (lower triangular) block operators. The sequence $w$ corresponds to the disturbance, $s$ is the state sequence and $v$ is the control sequence. A policy (controller) is a mapping from the sequence $w$ to the sequence~$v$. For a fixed policy, the quadratic cost is 
\begin{align}\label{eq:cost_op}
    \text{cost}_{OP}(\pi;w) \triangleq \|s\|_2^2 + \|v\|_2^2.
\end{align}

\rr{The regret problem for the general formulation in \eqref{eq:operator_sys}-\eqref{eq:cost_op} can be defined similarly to Problem \ref{prob:regret} as
\begin{equation}\label{eq:regret_OP_def}
\inf_{\pi\in\Pi^{\text{S.C.}}}\sup_{\|w\|_2\le1}\left( \text{cost}_{OP}(\pi;w)-\inf\limits _{\pi'\in\Pi^{\text{N.C.}}} \text{cost}_{OP}(\pi';w)\right).
\end{equation}
}
Note that the time-horizon here can be either finite, semi-infinite, or the doubly-infinite regime in \eqref{eq:cost_def_DoublyInf}.

The following result characterizes the optimal non-causal policy $\pi'\in\Pi^{\text{N.C.}}$ in \eqref{eq:regret_OP_def}, which is shown to be linear. 
\begin{theorem}[{The non-causal controller \cite[Th. $11.2.1$]{hassibi1999indefinite}}]  \label{th:noncausal}
The optimal non-causal controller \rr{that minimizes the quadratic cost in \eqref{eq:cost_op}} is linear and is given by $v = K_0w$, where $K_0$ is the linear operator 
\begin{align}\label{eq:non-causal_cont_op}
    K_0 = -(I + F^{*}F)^{-1}F^{*}G.
\end{align}
\end{theorem}
For completeness, Theorem \ref{th:noncausal} is proved in Appendix \ref{app:non-causal}. %\bb{\cite[Appendix~\ref{app:non-causal}]{extended}}. 
\rr{In Theorem \ref{th:linear_is_optimal} below, we show that the policy $\pi\in\Pi^{S.C}$ that minimizes the regret in the infinite-horizon regime of \eqref{eq:regret_OP_def} is also linear.} Thus, we focus on linear controllers so as to simplify the regret in \eqref{eq:regret_OP_def}. Consider a linear controller $K$, and its cost operator can be defined as
\begin{align}
    \left[ \begin{array}{c} s \\ v \end{array} \right] = \underbrace{\left[ \begin{array}{c} FK+G \\ K \end{array} \right]}_{\triangleq T_K} w
    \label{transfer_operator}
\end{align}
to compactly expressed its quadratic cost as 
$$\text{cost}_{OP}(K;w) = w^*T_K^*T_K w.$$

We can now use the completion of the square to express, for any linear controller $K$, its squared cost operator as
\begin{align}\label{eq:operator_equiv}
    T_K^\ast T_K & =  (K - K_0)^* (I+F^*F) (K - K_0) + T_{K_0}^*T_{K_0}
\end{align}
with
$$T_{K_0}^*T_{K_0} = G^*(I + FF^*)^{-1}G.$$
Note that \eqref{eq:operator_equiv} implies $T_K^\ast T_K\succeq T_{K_0}^*T_{K_0}$, so that the non-causal controller $K_0$ outperforms (in terms of LQR cost) any linear controller $K$ {\em for any} $w$.

It is interesting to compare the objectives of regret-optimal control and the classical robust control. In both formulations there is a maximization over $w$ that can be replaced with an operator norm, and their objectives are
\begin{align}\label{eq:comparison_robust}
    \underbrace{\inf_{\text{s. causal $K$}} \|T_K^*T_K\|}_{\mbox{$\mathcal H_\infty$ control}} ~~,~~
    \underbrace{\inf_{\text{s. causal $K$}} \|T_K^\ast T_K - T_{K_0}^\ast T_{K_0}\|}_{\mbox{regret-optimal control}}.
\end{align}
The difference is transparent; in $\mathcal H_\infty$ control, one aims to minimize the worst-case gain from the disturbance energy to the control cost, whereas in regret-optimal control one attempts to minimize the worst-case gain from the disturbance energy to the regret. \rr{This latter fact makes the regret-optimal controller competitive with respect to the non-causal controller since it has as its baseline the best that any controller can do, whereas the $\mathcal H_\infty$ controller has no baseline to measure itself against.} Comparison of regret-optimal control and $\mathcal H_2$ control is discussed after Theorem \ref{th:stricrly_regret}.

The state-space model is a special case of \eqref{eq:operator_sys}. Specifically, if we choose the causal operators $F$ and $G$ to be lower triangular, doubly-infinite block Toeplitz operators with Markov parameters $F_i = Q^{1/2}A^{i-1}B_uR^{-1/2}$ and $G_i = Q^{1/2}A^{i-1}B_w$, respectively, for $i>0$. Furthermore, the weight matrices $Q = Q^{*/2}Q^{1/2}$ and $R = R^{*/2}R^{1/2}$ can be absorbed in the transformed state and control variables~$s = \{Q^{1/2}x_t\}$ and $v = \{R^{1/2}u_t\}$, respectively. The LQR cost in \eqref{eq:cost_def_DoublyInf} is now in the required form $$\cost(\pi; w) = \|s\|_2^2 + \|v\|_2^2.$$ 
\rr{The regret problem in \eqref{eq:def_regret} is a special case of its operator formulation counterpart in \eqref{eq:regret_OP_def}. Combining \eqref{eq:regret_OP_def} with the optimality of the non-causal and causal controllers (shown below in Theorem \ref{th:linear_is_optimal}), we can write the regret in Problem \ref{prob:regret} as
\begin{align}
    \regret&= \inf_{\text{s. causal $K$}} \|T_K^\ast T_K - T_{K_0}^\ast T_{K_0}\|,
\end{align}
where $K$ is a linear controller.}

% Further, we can represent the linear causal policy $\pi$, mapping $w$ to $v$, by a causal (lower triangular) doubly infinite operator $K$, i.e., $v = Kw$. In this case, the transfer operator, $T_K$, mapping the disturbance $w$ to the sequences $s$ and $v$ is given by
% \begin{align}
%     \left[ \begin{array}{c} s \\ v \end{array} \right] = \underbrace{\left[ \begin{array}{c} FK+G \\ K \end{array} \right]}_{T_K} w.
%     \label{transfer_operator}
% \end{align}
% With this definition, the cost is just $T_K^\ast T_K$.
% We first derive the offline optimal control policy $\pi^*$ and show that it is a linear function of $w$. By this result, we can also provide an alternative expression for the squared operator of an arbitrary controller in terms of $K_0$.
\subsection{The Nehari problem}\label{subsec:preli_nehari}
Before proceeding to the main results, we present a problem fundamental to the solution of regret-optimal control.
\begin{problem} [\bf Nehari Problem \cite{nehari1957bounded}]\label{prob:nehari} Given a strictly anti-causal (strictly upper triangular) doubly-infinite block Toeplitz operator $U$, find a causal (lower triangular) doubly-infinite block Toeplitz operator $L$, such that $\|L-U\|$ is minimized.
\end{problem}
The Nehari problem seeks the best causal approximation to a strictly anti-causal operator in the operator norm sense. Nehari showed the minimal norm can be characterized by the Hankel norm of an operator \cite{nehari1957bounded}. As we will see in Theorem \ref{th:Nehari_general} and its application to our problem in the next section, when the operator has a state-space structure, the minimal norm and the approximation $L$ can be found explicitly.

% \bb{Throughout the paper, we will occasionally refer to a \emph{$\gamma$-optimal solution} for the Nehari problem, that is, a solution $L_\gamma$ that achieves $\|L_\gamma - U\|\le\gamma$, when such a solution exists.}

\section{Main results}\label{sec:main}
This section includes our main results. In Section \ref{subsec:main_nehari_reduction}, we present the regret as a Nehari problem \rr{and the optimality of linear controllers}. We then present the optimal regret value and the regret-optimal controller for the state-space setting.

\subsection{Reduction to a Nehari problem} \label{subsec:main_nehari_reduction}
% Now that we have determined that the optimal offline policy is linear, we can rewrite Problem \ref{prob:regret} as
% \begin{align}
%     \inf_{\text{causal} \ K} \sup_{ \|w\|_2 \leq 1} \left(w^*T_K^*T_Kw - w^*T_{K_0}^*T_{K_0}w\right) = \inf_{\mbox{causal $K$}} \sigma_{\text{max}}^{2}\left(T_K^*T_K - T_{K_0}^*T_{K_0}\right),
% \end{align}
% where $\sigma_{\text{max}}(\cdot)$ represents the operator norm.

%  We therefore have
% $$
% T_K^*T_K-T_{K_0}^*T_{K_0} = (\Delta K-\Delta K_0)^*(\Delta K-\Delta K_0),
% $$
% and therefore $\sigma_{\text{max}}(T_K^*T_K-T_{K_0}^*T_{K_0}) = \sigma_{\text{max}}^2(\Delta K-\Delta K_0)$. We are now in a position to state the Nehari problem.
The following theorem presents the relation between the Nehari problem and regret when restricted to linear controllers.
\begin{theorem}[Regret as a Nehari problem]\label{th:reg_as_Nehari}
The regret problem \rr{with linear controllers in \eqref{eq:comparison_robust}} can be formulated as the Nehari problem
\begin{align}\label{eq:th_op_nehari}
    \inf_{\mbox{s.causal $K$}} \|T_K^\ast T_K - T_{K_0}^\ast T_{K_0}\|&= \mspace{-10mu}\inf_{\mbox{s.causal $L$}}\|L - \left\{\Delta K_0\right\}_-\|^2,
\end{align}
where $\Delta$ is given by the canonical factorization~$\Delta^\ast\Delta = I+F^\ast F$, $K_0$ is the optimal non-causal controller \rr{in \eqref{eq:non-causal_cont_op},} and $\{\cdot\}_-$ denotes the anti-causal part of a linear operator.

Furthermore, let $L$ be a solution to the Nehari problem in \eqref{eq:th_op_nehari}, then a regret-optimal \rr{linear} controller is given by
\begin{align}\label{eq:th_reduction_con}
    K = \Delta^{-1}\left(L + \left\{\Delta K_0\right\}_+\right),
\end{align}
where $\{\cdot\}_+$ denotes the strictly causal part of an operator.
\end{theorem}
Nehari showed that the minimal value in \eqref{eq:th_op_nehari} is the Hankel norm of the anticausal operator $\left\{\Delta K_0\right\}_-$ \cite{nehari1957bounded}. It is rather involved to compute the norm and the causal operator $L$ unless the operator has a structure. In Section \ref{subsec:general_nehari}, we show that the Nehari problem in \eqref{eq:th_op_nehari} can be solved explicitly when the noncausal operator has a state-space structure which leads to explicit solution of the regret-optimal control problem. Theorem \ref{th:reg_as_Nehari} also reveals the steps required to derive the regret-optimal controller: a factorization of the positive operator $I+F^\ast F$ and a decomposition of $\Delta K_0 = \{\Delta K_0\}_+ + \{\Delta K_0\}_-$. 

To assess the difference between regret-optimal and $\mathcal H_2$ control, we express the non-causal controller as
\begin{align}
   \Delta^{-1}\Delta K_0  &= \Delta^{-1} (\{\Delta K_0\}_- + \{\Delta K_0\}_+).
\end{align}
It can be shown that the expression $\Delta^{-1}\{\Delta K_0\}_+$ is precisely the optimal law in $\mathcal H_2$ control (in the case of a state-space model, it is the LQR state-feedback law). In other words, in $\mathcal H_2$, the anticausal term $\{\Delta K_0\}_-$ is eliminated from the optimal controller. However, in regret-optimal control, the anti-causal term $\{\Delta K_0\}_-$ is approximated with a causal operator $L$ using a Nehari problem. This relation is translated in Theorem \ref{th:stricrly_regret} to show that the regret-optimal controller is a sum of the $\mathcal H_2$ controller and a controller driven by a \emph{Nehari solution}.  

In Theorem \ref{th:reg_as_Nehari}, the controllers are restricted to be linear. We proceed to show that there is no loss of optimality in considering linear controllers. 
\rr{\begin{theorem}[Optimality of linear controllers]\label{th:linear_is_optimal}
The linear regret-optimal controller in \eqref{eq:th_reduction_con} attains the optimal regret in~\eqref{eq:regret_OP_def}. That is, there is no loss optimality when restricting to optimize over linear policies:
\begin{align}
&\inf_{\pi\in\Pi^{\text{S.C.}}}\sup_{\|w\|_2\le1}\left( \text{cost}_{OP}(\pi;w)-\inf\limits _{\pi'\in\Pi^{\text{N.C.}}} \text{cost}_{OP}(\pi';w)\right)
\nn\\
&= \inf_{\emph{s.causal $K$}} \|T_K^\ast T_K - T_{K_0}^\ast T_{K_0}\|.
\end{align}
\end{theorem}
The proof of Theorem \ref{th:linear_is_optimal} may be of independent interest since its derivation relies on studying a non-linear generalization of the Nehari problem in which linear approximations are shown to be optimal. The proofs of Theorems \ref{th:reg_as_Nehari} and \ref{th:linear_is_optimal} appear in Section \ref{sec:proof}.}

\subsection{The state-space setting}\label{subsec:main_SS}
Section \ref{subsec:main_nehari_reduction} dealt with regret-optimal control for general operators. Here we focus on the state-space setting in order to obtain explicit results.

Define $P\succeq 0$ as the unique stabilizing solution to the LQR Riccati equation
\begin{align}\label{eq:Riccati}
  P = Q + A^\ast P A - A^\ast  P B_u (R + B_u^\ast P B_u ) ^{-1} B_u^\ast P A, 
\end{align}
$K_{lqr}= (R + B_u^\ast P B_u )^{-1} B_u^\ast P A$ is the LQR controller and \rr{$A_K \triangleq A-B_uK_{lqr}$ is the closed-loop system}. The following is our main result.

\begin{theorem}[Regret-optimal control]\label{th:stricrly_regret}
The optimal regret in \eqref{prob:regret} is
\begin{align}\label{eq:proof_Nehari_SC}
    {\regret}&=\lambda_{\text{max}}(Z{\Pi}),
\end{align}
where $Z$ and $\Pi$ solve the Lyapunov equations
\begin{align}\label{eq:SC_regret_Lyapunov}
    Z &= A_K Z A_K^\ast + B_u (R + B_u^\ast P B_u)^{-1} B_u^\ast\nn\\
    {\Pi} &= A_K^\ast {\Pi} A_K + P B_w B_w^{\ast} P.
\end{align}
A regret-optimal strictly causal controller is given by
\begin{align}\label{eq:th_sc_con}
u_t&=  \hat{u}_t - K_{\text{lqr}} x_t,
\end{align}
where $K_{lqr}$ is given in \eqref{eq:Riccati}, $\hat{u}_t$ is given by
\begin{align}\label{eq:th_SC_uhat}
        \xi_{t+1} &= {F}_\gamma \xi_t + {K}_\gamma w_t \nn\\
        \hat{u}_t&= - (R + B_u^\ast PB_u)^{-1}B_u^\ast {\Pi} \xi_t
\end{align}
with the constants
\begin{align}\label{eq:th_strictly_Nehari_SS}
    {K}_\gamma &=  (I - A_K Z_\gamma A_K^\ast {\Pi})^{-1}A_K Z_\gamma P B_w\nn\\
    {F}_\gamma &= A_K - {K}_\gamma B_w^\ast P,
\end{align}
and $Z_\gamma$ is the solution to the Lyapunov equation 
\begin{align}\label{eq:Z_gamma}
 Z_\gamma &= A_K Z_\gamma A_K^\ast + \gamma^{-2}B_u (R + B_u^\ast P B_u)^{-1} B_u^\ast.
\end{align}
\rr{with $\gamma^2=\regret$}. 
\end{theorem}
The explicit regret formula follows from our explicit solution to the Nehari problem in the state-space setup (Theorem \ref{th:Nehari_general}). The proof of Theorem \ref{th:stricrly_regret} appears in Section \ref{subsec:proof_ss}.
\rr{A computational advantage of regret-optimal control is that the optimal regret can be computed explicitly. This is in contrast to $\mathcal H_\infty$ control where indefinite factorizations are needed, and whose solution relies on a bisection method to determine the minimal norm in \eqref{eq:comparison_robust}.} Also, note that the solution only requires a solution to the standard LQR Riccati equation with two additional Lyapunov equations to obtain $Z_\gamma$ and $\Pi$.

Recall the optimal $\mathcal H_2$ (LQR) controller
\begin{align}\label{eq:state-feedback}
    u^{H_2}_t&= - K_{\text{lqr}} x_t.
\end{align}
The regret-optimal controller is the sum of \eqref{eq:state-feedback} and an additional state $\hat{u}_t$ driven by the state-space in \eqref{eq:th_SC_uhat}. This implies that, in contrast to $\mathcal H_2$ control, the regret-optimal controller depends on all past disturbances (or states).
\rr{
\begin{remark}\label{remark:states}
From a practical point of view, the strictly-causal controller in \eqref{eq:th_sc_con} can be implemented as a function of the system states without access to the underlying disturbances. In particular, note that $w_t$ only appears in \eqref{eq:th_SC_uhat} and $${K}_\gamma w_{t} = (I - A_K Z_\gamma A_K^\ast {\Pi})^{-1}A_K Z_\gamma P (x_{t+1} - A x_{t} - B_u u_{t}).$$
\end{remark}}

\begin{figure*}[t]
    \centering
    \begin{subfigure}[t]{0.3\textwidth}
\includegraphics[width=1\textwidth]{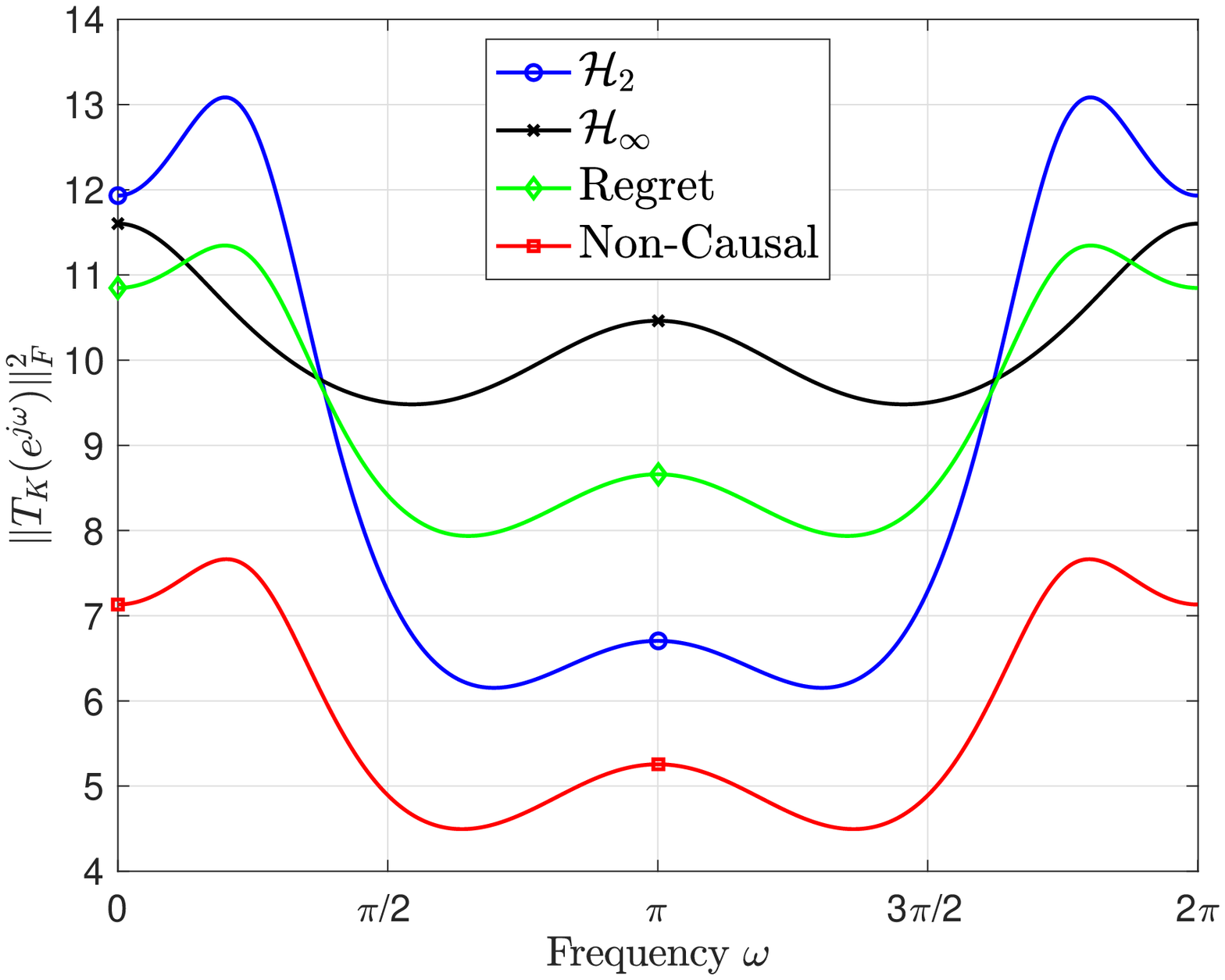}
\caption{Squared Frobenius Norm}
\end{subfigure}
\hspace{1em}
    \begin{subfigure}[t]{0.3\textwidth}
\includegraphics[width=1\textwidth]{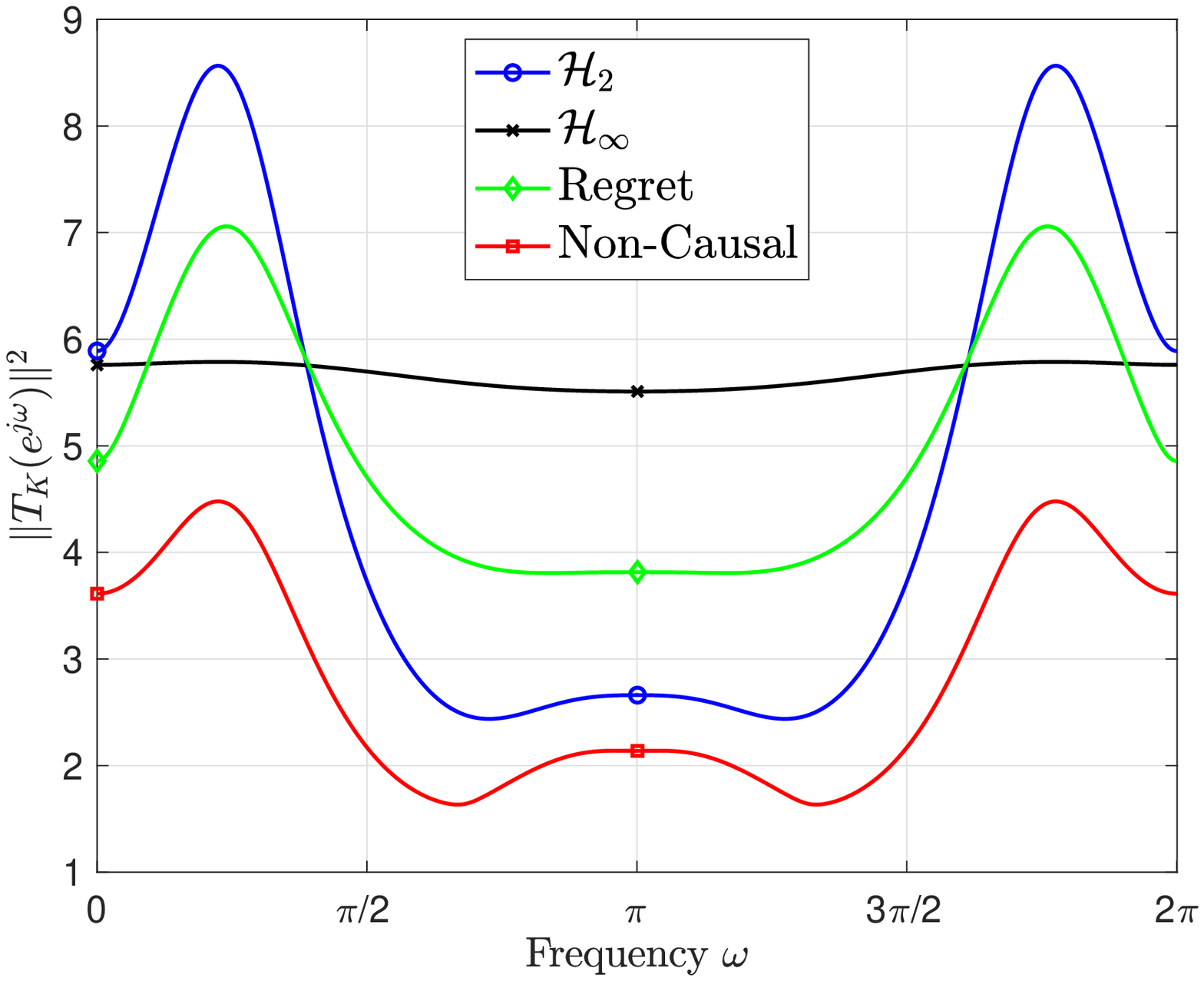}
\caption{Squared Operator Norm}
\end{subfigure}
\hspace{1em}
 \begin{subfigure}[t]{0.3\textwidth}
\includegraphics[width=1\textwidth]{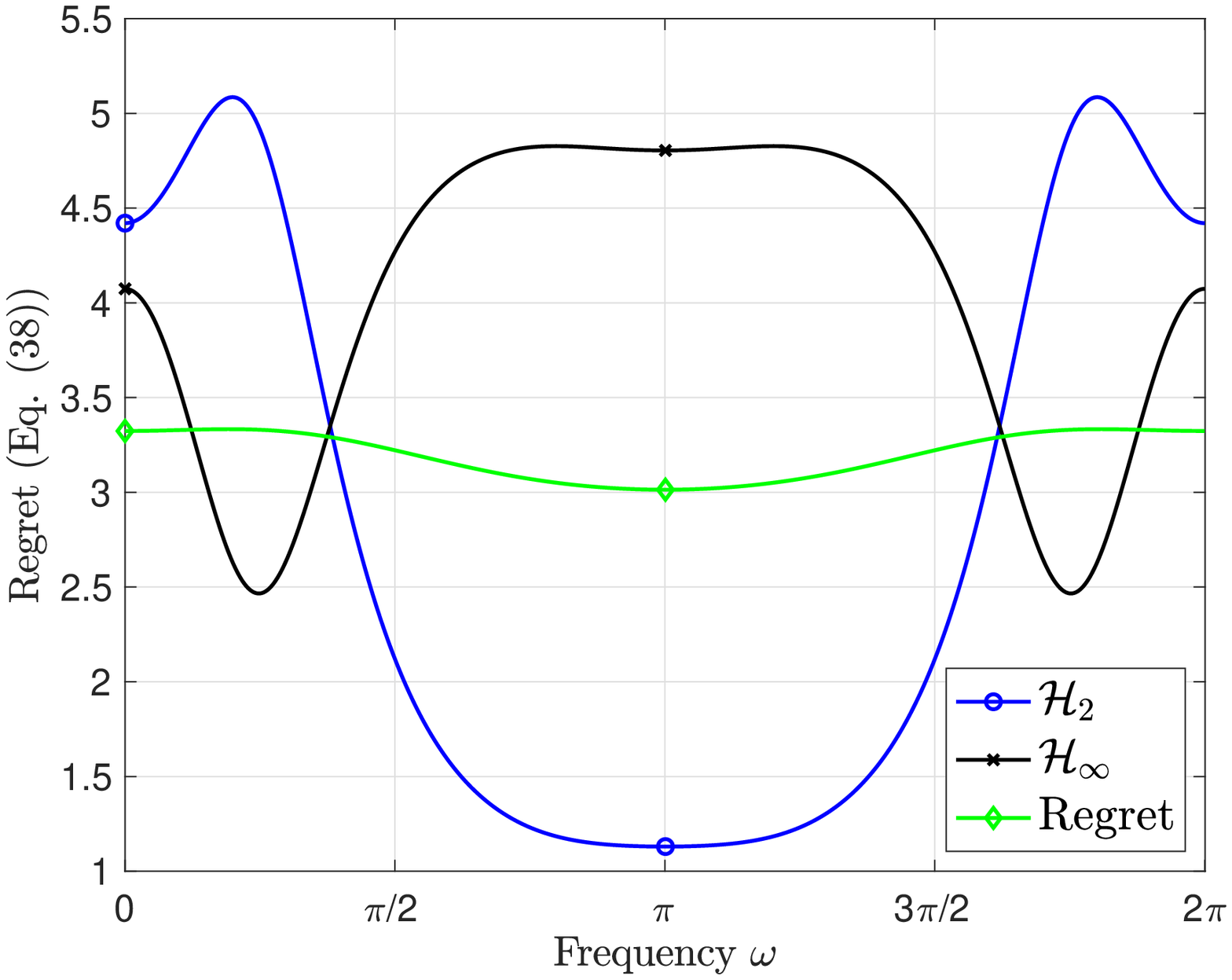}
\caption{Regret}
\end{subfigure}
    \caption{
    Performance metrics on $T_K(e^{j\omega})$ for a random system.}
    \label{fig:random}
    \vspace{-1.5em}
\end{figure*}

\section{Numerical Simulations} \label{sec:numerical}
In this section, we present the performance of the regret-optimal, the $\mathcal H_2$, and the $\mathcal H_\infty$ controllers for different systems. We present a frequency domain evaluation so as to compare the controllers across the full range of disturbances. We then show time-domain evaluations with various disturbances.
\subsection{Frequency-domain}\label{subsec:ex_freq}
The performance of any (linear) controller is governed by the transfer operator $T_K$ in \eqref{transfer_operator}. 
% It is useful to represent this operator via its transfer function in the $z$-domain, i.e.,
% \begin{align}\label{eq:ex_TKz}
% T_K(z) &= \left[\begin{array}{c} F(z)K(z)+G(z) \\ K(z) \end{array} \right],    
% \end{align}
The squared Frobenius norm of $T_K$, which is what $\mathcal H_2$ controller minimizes, is given by\looseness=-1
\begin{align}\label{eq:ex_TkFrob}
\|T_K\|_F^2 &= \frac{1}{2\pi}\int_0^{2\pi}\mbox{trace}\left(T_K^*(e^{j\omega})T_K(e^{j\omega})\right)d\omega,
\end{align}
the squared operator norm of $T_K$, which is what $\mathcal H_\infty$ controller minimizes, is given by
\begin{align}\label{eq:ex_Tknorm}
\|T_K\|^2 &= \max_{0\leq\omega\leq 2\pi} \sigma_{\max}\left(T_K^*(e^{j\omega})T_K(e^{j\omega})\right),
\end{align}
and the regret-optimal controller minimizes
\begin{align}\label{eq:ex_regret}
&\left\|T_K^*T_K-T_{K_0}^*T_{K_0}\right\| \\
&= \max_{0\leq\omega\leq 2\pi}\sigma_{\max}\left(T_K^\ast (e^{j\omega})T_K(e^{j\omega})-T_{K_0}^\ast(e^{j\omega})T_{K_0}(e^{j\omega})\right). \nn
\end{align}
To assess and compare the performance of the different controllers across the range of disturbances, we plot the arguments of the metrics in \eqref{eq:ex_TkFrob}-\eqref{eq:ex_regret} as a function of frequency $\omega$. We first consider a random time-invariant linear dynamical system with $n=6$ and $m=2$. All matrices, \textit{i.e.}, $A, B_u, B_w, Q, R$, are randomly generated and $A$ is unstable but $(A, B_u)$ is stabilizable. We construct the optimal non-causal, $\mathcal H_2$, $\mathcal H_\infty$, and regret-optimal controllers. Fig. \ref{fig:random} describes the performance of the various controllers.

\begin{table*}[b]
 \rr{
  \caption{Performance of the controllers in different systems. The best performance in each metric is highlighted\label{table_all}}}    
\resizebox{0.99\textwidth}{!}{{\begin{tabular}{ |c| 
c c c| c c c | c c c | 
 c c c | c c c | c c c | c c c
}
  &  
  & \textbf{HE1\cite{leibfritz2003description}} &  & & \textbf{AC15\cite{leibfritz2003description}} & &  & \textbf{REA1\cite{leibfritz2003description}} & \\ \hline
                  & $\mathbf{\|T_K\|_F^2}$ & $\mathbf{\|T_K\|^2}$ & \textbf{Regret} & $\mathbf{\|T_K\|_F^2}$ & $\mathbf{\|T_K\|^2}$ & \textbf{Regret} & $\mathbf{\|T_K\|_F^2}$ & $\mathbf{\|T_K\|^2}$ & \textbf{Regret}\\
                %   & $\mathbf{\|T_K\|_F^2}$ & $\mathbf{\|T_K\|^2}$ & \textbf{Regret}\\ 
                  \hline 
 Noncausal & 
%   $1.86\!\times\!10^1$ & $2.82\!\times\!10^1$ & 0 &
  $0.40\!\times\!10^0$ & $8.99\!\times\!10^1$ & ${0}$ & $7.29\!\times\!10^3$ & $1.46\! \times\!10^6$ & ${0}$ & $5.18\!\times\!10^1$ &  $2.05\!\times\!10^3$ & ${0}$\\
  
  Regret-optimal & 
%   $3.70\!\times\!10^1$ & $3.19\!\times\!10^1$& $2.22\!\times\!10^2$&
  $7.19\!\times\!10^1$ &  $1.61\!\times\!10^2$ &  $\mathbf{7.23\!\times\!10^1}$ &  $1.88 \!\times\! 10^4$ & $2.28\!\times\!10^6$ & $\mathbf{9.55\!\times\!10^5}$ & $3.38\!\times\!10^3$ & $5.30\!\times\!10^3$ & $\mathbf{3.32\!\times\!10^3}$
  \\ 
  
  $\mathcal H_2$ &
%   $3.08\!\times\!10^1$ & $3.71\!\times\!10^1$& $5.77\!\times\!10^2$&
  $\mathbf{1.09\!\times\!10^0}$  &  $3.11\!\times\!10^2$  & $2.21\!\times\!10^2$ & $\mathbf{1.61\!\times\!10^4}$& $2.72\!\times\!10^6$& $1.40\!\times\!10^6$& $\mathbf{2.62\!\times\!10^2}$   & $1.46\!\times\!10^4$ & $1.26\!\times\!10^4$  \\
  
 $\mathcal H_\infty$ &
%  $8.62\!\times\!10^1$ & $2.82\!\times\!10^1$ & $5.51\!\times\!10^2$ &
 $1.31\!\times\!10^2$ & $\mathbf{1.31\!\times\!10^2}$  & $1.31\!\times\!10^2$ & $1.41\!\times\!10^6$ & $\mathbf{2.19\!\times\!10^6}$ & $2.20\!\times\!10^6$ & $4.40\!\times\!10^3$  & $\mathbf{4.36\!\times\!10^3}$  & $4.36\!\times\!10^3$
 \\
 
 \hline 
%  \hline & & \textbf{REA1\cite{leibfritz2003description}} & & & 
% %  \textbf{AGS\cite{leibfritz2003description}} & & & 
%  \textbf{WEC1\cite{leibfritz2003description}} & & & \textbf{AC12\cite{leibfritz2003description}} & \\ \hline
%                   & $\mathbf{\|T_K\|_F^2}$ & $\mathbf{\|T_K\|^2}$ & \textbf{Regret} & $\mathbf{\|T_K\|_F^2}$ & $\mathbf{\|T_K\|^2}$ & \textbf{Regret} & $\mathbf{\|T_K\|_F^2}$ & $\mathbf{\|T_K\|^2}$ & \textbf{Regret} \\
%                 %   & $\mathbf{\|T_K\|_F^2}$ & $\mathbf{\|T_K\|^2}$ & \textbf{Regret}\\ 
%                   \hline
                  
% Noncausal &  $5.18\!\times\!10^1$ &  $2.05\!\times\!10^3$ & 0 & $4.23\!\times\!10^3$  & $4.01\! \times\!10^5$ & 0 & $6.65\!\times\!10^2$ & $8.29\! \times\!10^4$ & 0 \\
  
%   Regret-optimal & $3.38\!\times\!10^3$ & $5.30\!\times\!10^3$ & ${3.32\!\times\!10^3}$ &  $1.59 \!\times\! 10^5$  & $5.53\!\times\!10^5$ & ${1.56\!\times\!10^5}$ & $3.99 \!\times\! 10^4$ & $1.21\!\times\!10^5$ & ${3.97\!\times\!10^4}$ \\
  
%   $\mathcal H_2$ & ${2.62\!\times\!10^2}$   & $1.46\!\times\!10^4$ & $1.26\!\times\!10^4$  &  ${1.04\!\times\!10^4}$ & $9.59\!\times\!10^5$ & $5.65\!\times\!10^5$ &  ${2.76 \!\times\!10^3}$ & $1.92\!\times\!10^5$ & $1.27\!\times\!10^5$ \\

%  $\mathcal H_\infty$ & $4.40\!\times\!10^3$  & ${4.36\!\times\!10^3}$  & $4.36\!\times\!10^3$  & $4.20\!\times\!10^5$ & ${4.19\!\times\!10^5}$ & $4.19\!\times\!10^5$ & $8.29\!\times\!10^4$ & ${8.29\!\times\!10^4}$ & $8.29\!\times\!10^4$ \\ \hline
\end{tabular}}}
\vspace{-1em}
\end{table*}

\begin{figure*}[h]
% \captionsetup{labelfont=bf}
    \centering
    \begin{subfigure}[t]{0.31\textwidth}
\includegraphics[width=1.1\textwidth]{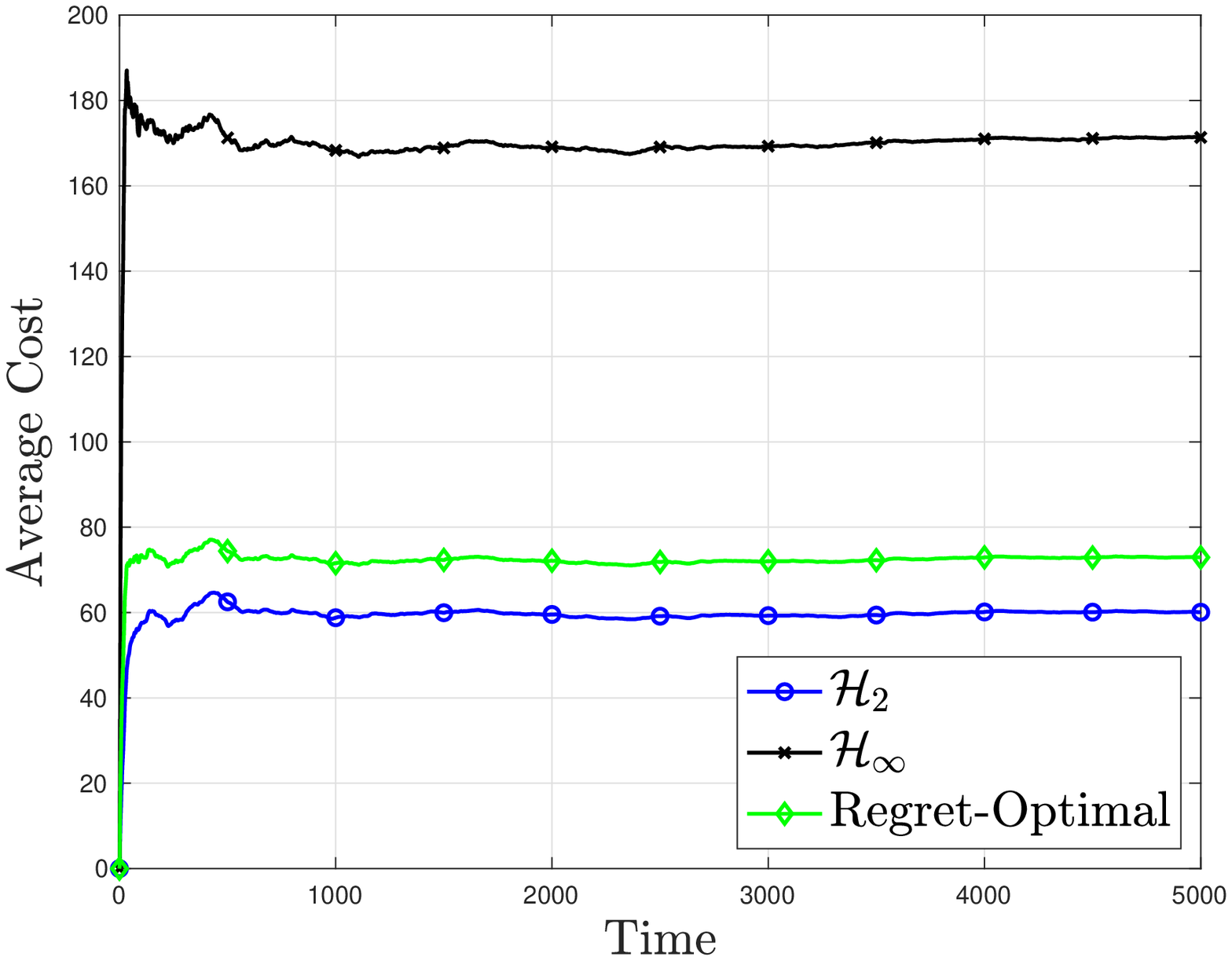}
\caption{White Noise}
\end{subfigure}
\begin{subfigure}[t]{0.31\textwidth}
\includegraphics[width=1.1\textwidth]{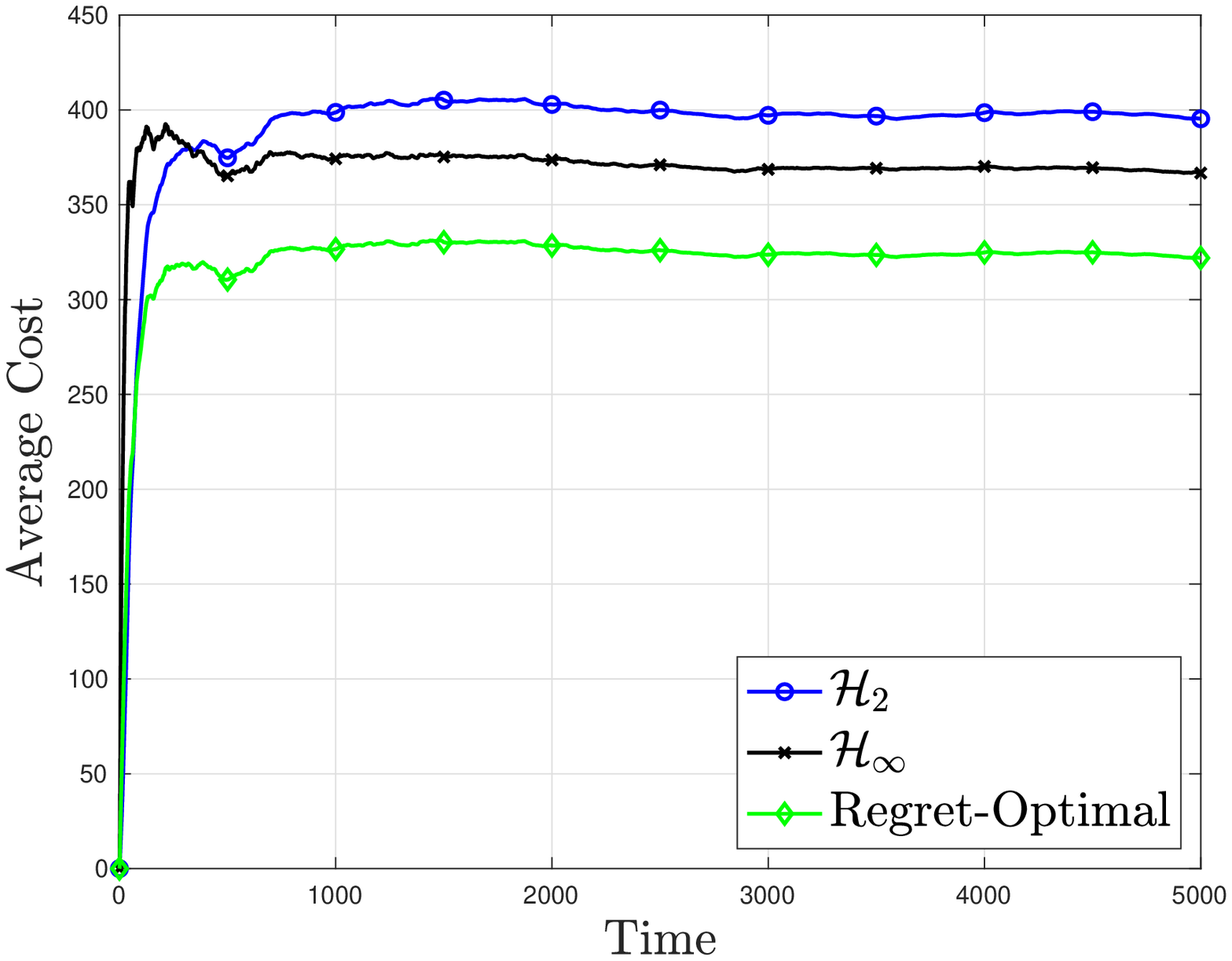}
\caption{White Noise with a small constant }
\end{subfigure}
\begin{subfigure}[t]{0.31\textwidth}
\includegraphics[width=1.1\textwidth]{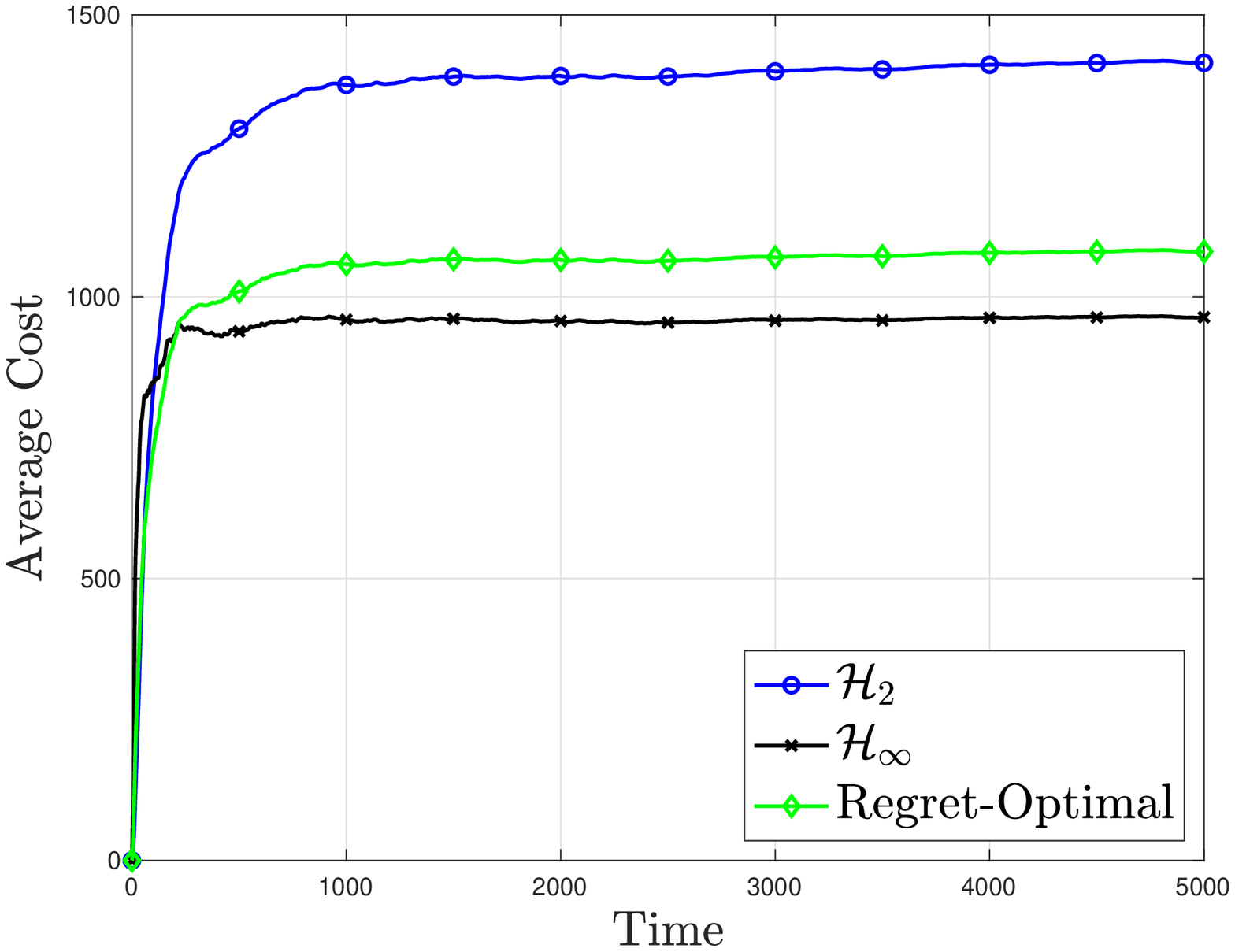}
\caption{White Noise with a large constant}
\end{subfigure}
\caption{The control cost of the Boeing 747 example under different disturbances. \textbf{(a)} The $\mathcal H_2$ controller outperforms other causal controllers. Note that the regret-optimal controller has a comparable performance with the $\mathcal H_2$ controller. \textbf{(b)} In this regime of a small DC component, the regret-optimal controller outperforms all causal controllers. \textbf{(c)} In this regime of a larger DC component, $\mathcal H_\infty$ controller outperforms all causal controllers. Note that regret-optimal performs noticeably close to $\mathcal H_\infty$.}
\label{fig:time_domain}
\end{figure*}

It is evident in Fig. \ref{fig:random}(a)-(b) that the non-causal controller outperforms the causal controllers in all metrics. The $\mathcal H_2$ controller attains the minimal Frobenius norm, which is the area under its curve in Fig. \ref{fig:random}(a) (it can be also viewed as the minimal expected cost with iid $w$). However, in doing so, it has relatively large cost for low frequencies and the highest peak. The $\mathcal H_\infty$ is a robust controller designed to minimize the operator norm, which is the peak of the per-frequency norm, Fig. \ref{fig:random}(b). However, in doing so, it sacrifices the average performance and has a relatively large area under the curve in Fig. \ref{fig:random}(a). Regret-optimal controller finds the best of both worlds and lies in between the performance of $\mathcal H_2$ and $\mathcal H_\infty$ or below, \textit{i.e.}, stays close to the best average performance of $\mathcal H_2$ and its peak is relatively close to the best peak of $\mathcal H_\infty$.\looseness=-1

Recall that the regret-optimal controller aims to stay as close as possible to the non-causal controller across all frequencies \rr{by minimizing its largest deviation from the latter}. In doing so, it achieves an area under the curve that is close to that of the $\mathcal H_2$-optimal controller ($9.34$ vs. $8.85$) and it has a peak that improves significantly upon the peak of the $\mathcal H_2$ controller. This demonstrates that the regret-optimal controller interpolates between the robust and the average performance of the $\mathcal H_2$ controller. It is interesting to compare the regret metric across the disturbances in Fig. \ref{fig:random}(c). Consistent with our theoretical claim, the regret-optimal control attains the smallest peak in the regret plot. Moreover, it maintains an almost-constant distance from the non-causal controller. This stands in contrast to the $\mathcal H_2$ and $\mathcal H_\infty$ controllers that may be closer to the non-causal for some frequencies but suffer suffer high regret in other regimes. This reveals the regret-optimality feature, by minimizing the largest cost deviation from the non-causal controller, it attains a balanced behavior across all input disturbances with respect to the (universal) benchmark. 

\rr{The behaviour illustrated above is indeed maintained in all studied examples including practical systems. We consider three linear time-invariant models: an aircraft dynamic (AC15), an helicopter model (HE1), and a chemical reactor model (REA1)~\cite{leibfritz2003description}. These models cover a wide range of applications and have various state and control input dimensions. Further details can be found in \cite{leibfritz2003description}. The different norms resulting for each controller are given in Table \ref{table_all}. It can be observed that in terms of the Frobenius norm, the regret-optimal controller improves the performance of the $\mathcal H_\infty$ controller while improving the performance of the $\mathcal H_2$ in terms of the operator norm.\looseness=-1}

\subsection{Time-domain evaluation}\label{subsec:ex_time}
In this section, we illustrate the performance of the regret-optimal controller in time-domain. We consider the longitudinal flight control of Boeing 747 with dynamics linearized at an altitude of $40000$ft with a speed of $774$ft/sec, with 1-second discretization. The parameters of the system dynamics are provided in \cite{boyd2018introduction}. 
% For level flight of Boeing $747$ at the altitude of 40000ft with the speed of 774ft/sec, and a discretization every 1 second, the dynamics can be represented with a linear dynamical system with 
% \begin{equation*}
%     A = \begin{bmatrix}
%     .99 & .03 & -.02 & -.32 \\
%     .01 & .47 & 4.7 & 0 \\
%     .02 & -.06 & .40 & 0 \\
%     .01 & -.04 & .72 & .99 
% \end{bmatrix} \quad 
% B_u = \begin{bmatrix}
%   0.01 & 0.99 \\
%   -3.44 & 1.66 \\
%   -0.83 & 0.44 \\
%   -0.47 & 0.25
% \end{bmatrix}
% \end{equation*}
% and $B_w = I$. 
We select $Q = I$ and $R = I$. %For this dynamical system, we construct the regret-optimal, the $\mathcal H_2$, and the $\mathcal H_\infty$ controllers. 

In Fig. \ref{fig:time_domain}, we present the cost attained by the different controllers with different noise disturbances. In all experiments, results are averaged over 30 independent trials. In Fig. \ref{fig:time_domain}\textbf{(a)}, the disturbance is a white Gaussian noise. As expected, the $\mathcal H_2$ controller outperforms all causal controllers. It can be noted that the regret-optimal controller has comparable performance to the $\mathcal H_2$, and significantly improves the cost of the $\mathcal H_\infty$. Next, we choose the disturbance as a white noise with a constant (i.e., zero-frequency DC signal). In this case, the power ratio between the Gaussian noise govern the performance. In Fig. \ref{fig:time_domain}\textbf{(b)}, an evaluation with \emph{small} DC component is presented: we extract the eigenvector that corresponds to the largest singular value of $T_K(e^{j\omega}=1)$ (for $\mathcal{H}_2$) and add it to a white Gaussian noise. Fig. \ref{fig:time_domain}\textbf{(b)} shows that the regret-optimal controller outperforms the $\mathcal H_2$ and $\mathcal H_\infty$ controllers. It can be also seen that the $\mathcal H_2$ and the $\mathcal H_\infty$ are close in their performance, a gap that will diminish if the DC weight is growing large. In Fig. \ref{fig:time_domain}\textbf{(c)}, the other scenario is evaluated when the power of the DC component is doubled and governs the white noise. In this scenario, the robust $\mathcal H_\infty$ controller outperforms the other controllers whereas $\mathcal H_2$ performs poorly. However, the regret-optimal controller performs close to $\mathcal H_\infty$ controller which demonstrates the best of both worlds behavior of the regret-optimal controller.\looseness=-1

% \begin{figure}[t]
%     \centering
%     \includegraphics[width=0.49\textwidth]{figures/high_dc_avg.eps}
%     \caption{Strong-DC Time-domain evaluation of the numerical example: The system is excited with a white Gaussian noise that is added to a DC component, i.e., a constant. In this regime, the weight on the constant is \emph{larger}, which causes $\mathcal H_\infty$ to outperform all causal controller. Note that regret-optimal performs noticeably close to $\mathcal H_\infty$ controller.}
%     \label{fig:time_domain_3}
% \end{figure}

\section{Proofs}\label{sec:proof}
In this section, we present: a reduction of the regret to a Nehari problem (Th. \ref{th:reg_as_Nehari}), our solution to the general Nehari problem (Sec. \ref{subsec:general_nehari}), and a proof of Theorem \ref{th:stricrly_regret} along with the required technical lemmas (Sec. \ref{subsec:proof_ss}).

\subsection{Regret as a Nehari problem (Proof of Theorem \ref{th:reg_as_Nehari})}\label{sec:operator_proof}
\rr{Recall that we focus here on linear controllers, consider}
\begin{align}\label{eq:proof_reduction_main}
   &\inf_{\text{s. causal} \ K} \sup_{ \|w\|_2 \leq 1} \left(w^*T_K^*T_Kw - w^*T_{K_0}^*T_{K_0}w\right)\nn\\
    &= \inf_{\text{s. causal} \ K} \|T_K^*T_K-T_{K_0}^*T_{K_0}\| \nn\\
    &\stackrel{(a)}= \inf_{\text{s. causal} \ K} \|(K - K_0)^* (I+F^*F) (K - K_0)\| \nn\\
    &\stackrel{(b)}= \inf_{\text{s. causal} \ K} \|\Delta K-\Delta K_0\|^2 \nn\\
    % & = \inf_{\text{causal} \ K} \|\Delta K-\left\{\Delta K_0\right\}_+-\left\{\Delta K_0\right\}_-\|^2\nn\\
    &\stackrel{(c)}= \inf_{\text{s. causal} \ L} \|L-\left\{\Delta K_0\right\}_-\|^2,
\end{align}
where $(a)$ follows from \eqref{eq:operator_equiv}, $(b)$ follows from the canonical factorization $I+F^*F = \Delta^*\Delta$
with causal $\Delta$ and $\Delta^{-1}$ is causal and bounded by the positive definiteness of $I+F^*F$. Step $(c)$ follows from $\Delta K_0 = \{\Delta K_0\}_+ + \{\Delta K_0\}_-$, and the fact that, for any strictly causal $L$, one can recover a strictly causal $K$ by setting $K = \Delta^{-1}L + \{\Delta K_0\}_+$.

\subsection{\rr{Optimality of linear controllers (Theorem \ref{th:linear_is_optimal})}}\label{sec:proof_linear}
In this section, we prove Theorem \ref{th:linear_is_optimal}. The proof is based on showing that there exists no nonlinear controller that can achieve a regret better than the optimal linear controller.
\begin{proof}[Proof of Theorem \ref{th:linear_is_optimal}]
Recall the regret problem in \eqref{eq:regret_OP_def} 
\begin{equation}\label{eq:proof_lin_regret}
\min_{\pi\in\Pi^{\text{S.C.}}}\max_{w\in\ell_2}\frac{\text{cost}_{OP}(\pi_c;w)-\text{cost}_{OP}(K_0;w)}{\|w\|_2^2},
\end{equation}
where $K_0$ is the non-causal policy in Theorem \ref{th:noncausal}, $\Pi^{S.C.}$ is the set of (possibly nonlinear) strictly-causal policies $\pi: \ell_2\rightarrow\ell_2$ that can be represented as $v_i = \pi_{i}(w_{i-1},\ldots), \forall i$.
% Recall from \eqref{eq:cost_op} that
% \begin{eqnarray*}
% \text{cost}_{OP}\left(v = \pi_c(w),w\right) & = & \|v\|^2+\|Fv+Gw\|^2 \end{eqnarray*}
Similar to \eqref{eq:operator_equiv}, a completion-of-squares for the numerator of \eqref{eq:proof_lin_regret} gives that \eqref{eq:proof_lin_regret} can be written as
\begin{equation}
\min_{\pi\in\Pi^{S.C.}}\max_{w\in\ell_2}\frac{\|\Delta v - \Delta K_0 w\|^2}{\|w\|_2^2}.
\label{regret_optimal_1}
\end{equation}
The fact that $\Delta$ is causal and causally invertible implies that $v' = \Delta v$ is a causal mapping of $w$. Thus, (\ref{regret_optimal_1}) equals
\begin{equation}
\min_{\pi\in\Pi_c}\max_{w\in\ell_2}\frac{\|v'- \Delta K_0w\|^2}{\|w\|_2^2} = \gamma_{opt}^2.
\label{regret_optimal_2}
\end{equation}
It is not clear how to directly solve \eqref{regret_optimal_2} since $\pi_c$ is non-linear. Let us therefore focus on a suboptimal problem where, for a fixed $\gamma$, we ask whether there exists a policy $\pi$ such that
\begin{equation}
\max_{w\in\ell_2}\frac{\|v' - \Delta K_0w\|^2}{\|w\|^2} \leq \gamma^2.
\label{subopt}
\end{equation}
The smallest value of $\gamma$ for which \eqref{subopt} holds is the regret. Defining the non-causal operator $S = - \Delta K_0$, note further that (\ref{subopt}) is equivalent to
\begin{equation}
\|v'+Sw\|^2\leq\gamma^2\|w\|^2,~~~~\forall w\in\ell_2.
\label{subopt2}
\end{equation}
This optimization can be thought of as a generalized Nehari problem where the causal linear operator in Problem \ref{prob:nehari} is replaced with a non-linear causal policy. %In particular, we aim to find a causal policy $\pi_c$ that minimizes its norm when compared to the sequence $ Sw$ generated by the non-causal, linear operator $S$.

To see the effect of the causality constraint on $v$ it will be useful to introduce a partitioning of the infinite sequences and operators into "past" and "current and future" components. Thus, partition the sequences $v$ and $w$ as
$ v = \left[\begin{array}{c} v_- \\ v_+ \end{array} \right]$ and $w = \left[\begin{array}{c} w_- \\ w_+ \end{array} \right]$ where the semi-infinite sequences $v_- = \{\ldots ,v_{-2},v_{-1}\}$ and $w_- = \{\ldots ,w_{-2},w_{-1}\}$ represent the past and the semi-infinite sequences $v_+ = \{v_0,v_1,\ldots\}$ and $w_+ = \{w_0,w_1,\ldots\}$ represent the current and future. This partitioning induces the following partitioning on $S$
\begin{equation}
S = \left[\begin{array}{cc} S_- & S_A \\ S_H & S_+ \end{array} \right].
\end{equation}
The semi-infinite operators $S_-$ and $S_+$ map the past to past and current and future to current and future, respectively, and are called {\em Toeplitz} operators. The semi-infinite operators $S_H$ and $S_A$ map the past to current and future and current and future to past, respectively, and are called {\em Hankel} operators. The operator $S$ is non-causal and, therefore, the Hankel operator $S_A$ is non-zero.

With this partitioning, (\ref{subopt2}) can be rewritten as
\[ \left\|\left[\begin{array}{c} v'_-  +S_-w_-+S_Aw_+ \\ v'_++S_Hw_-+S_+w_+ \end{array} \right]\right\|^2\leq \gamma^2\left\|\left[\begin{array}{c} w_- \\ w_+ \end{array} \right]\right\|^2,~~~~  \forall w\in\ell_2.\]
Let us now focus on a causal disturbance $w$, i.e., $w_- = 0$. Since the policy $\pi_c$ is causal, this implies that $v_- = 0$ (since the control cannot react to future disturbances). Therefore the above inequality specializes to
\[ \left\|\left[\begin{array}{c} S_Aw_+ \\ v'_++S_+w_+ \end{array} \right]\right\|^2\leq \gamma^2\left\|\left[\begin{array}{c} 0 \\ w_+ \end{array} \right]\right\|^2,~~~~  \forall w_+\in\ell_{2,+}.\]
In other words, $\|v'_++S_+w_+\|^2 \!\leq\! w_+^\ast(\gamma^2I-S_A^\ast S_A)w_+, \forall w_+\in\ell_{2,+}$. 
% \[ \|v'_++S_+w_+\|^2 \leq w_+^\ast(\gamma^2I-S_A^\ast S_A)w_+, \ \ \ \forall w_+\in\ell_{2,+}. \]
For this inequality to hold true, at the very least the RHS must be non-negative for all $w_+\in\ell_{2,+}$. But this means that $\gamma^2I-S_A^\ast S_A\succeq 0$, or equivalently, that $\gamma^2\geq \sigma_{max}^2(S_A)$. In other words, the optimal regret of any nonlinear causal policy cannot be better than the squared maximal singular value of the Hankel operator $S_A$. But this is {\em precisely} the regret achieved by the optimal linear controller using the solution to the Nehari problem in \eqref{eq:proof_reduction_main}. Thus, in terms of minimizing regret, nonlinear policies offer no advantage over linear ones. 
% \bb{Now define the operator $L = \Delta K-\left\{\Delta K_0\right\}_+$. Note that $L$ is causal and bounded iff $K$ is causal and bounded. For one direction, assume that $K$ is causal and bounded. In that case, $L$ is clearly causal and bounded since it is the sum of two causal and bounded operators. For the other direction, assume that $L$ is causal and bounded and solve for $K$ to obtain $K = \Delta^{-1}\left(L+\left\{\Delta K_0\right\}_+\right)$. Since $\Delta^{-1}$ is causal and bounded, $K$ is also the sum of causal and bounded terms. This means that optimizing over $L$ and $K$ are equivalent. But optimizing over $L$ is just the Nehari problem. $K$ is then found from $L$.}
\end{proof}

% \vspace{-0.6em}

\subsection{The Nehari problem: a general solution}\label{subsec:general_nehari}
The following theorem summarizes our solution to the Nehari problem in the general case.
\begin{theorem}[The Nehari problem]\label{th:Nehari_general}
Consider the Nehari problem with $T(z) = H (z^{-1}I-F)^{-1}G$ in a minimal form and stable $F$. Then, the optimal norm is given by
\begin{align}\label{eq:th_general_nehari}
    \min_{\mbox{causal, bounded\ $L(z)$}} \| L(z)- T(z) \|^2&= \lambda_{\text{max}}(Z\Pi),
\end{align}
where $Z$ and $\Pi$ are the unique solutions to the Lyapunov equations
\begin{align}
 Z &= F^\ast Z F + H^\ast H\nn\\
\Pi &= F \Pi F^\ast + G G^\ast.
\end{align}

Moreover, an optimal solution to \eqref{eq:th_general_nehari} is given by
\begin{align}
    L(z)& = H \Pi (I + F_\gamma(zI- F_\gamma)^{-1} )K_\gamma,
\end{align}
with
\begin{align}
    K_\gamma &=  ( I - F^\ast Z_\gamma F \Pi  )^{-1}F^\ast Z_\gamma G\nn\\
    F_\gamma &= F^\ast - K_\gamma G^\ast,
\end{align}
and $Z_\gamma$ is the solution to the Lyapunov equation
\begin{align}
     Z_\gamma &= F^\ast Z_\gamma F + \gamma^{-2}H^\ast H.
\end{align}
with $\gamma^2 = \lambda_{\max}(Z\Pi)$.
\end{theorem}
Explicit solutions to the \emph{state-space Nehari problem} are known, e.g., \cite{Ball1990}\cite{GloverDoyle_book} but, to the best of our knowledge, the explicit solution in Theorem \ref{th:Nehari_general} is new and may be of independent interest. The derivation is based on the centralized solution obtained from a parameterization of all solutions to the Nehari problem \cite{hassibi1999indefinite}. The proof appears in Appendix \ref{app:Nehari_general}.%\bb{\cite[App. ?]{extended}}.

There is a slight difference between the Nehari problem in Theorem \ref{th:Nehari_general} and the one we aim to solve in Theorem \ref{th:reg_as_Nehari}. In Theorem \ref{th:Nehari_general}, we approximate a \emph{strictly} anticausal operator $T(z)$, but Theorem \ref{th:reg_as_Nehari} includes an anticausal operator. The required adaptation follows naturally in the $z$-domain,
\begin{align}
    \regret &= \min_{\text{s.causal} \ L(z)}  \| L(z) - T(z)\|^2\nn\\
    &=  \min_{\text{s.causal} \ L(z)} \| z L(z) - z T(z)\|^2 \nn\\
    &\stackrel{(a)}=  \min_{\text{causal} \ L'(z)} \| L'(z) - z T(z)\|^2,
\end{align}
where $(a)$ is due to the invertible substitution $L'(z) = zL(z)$, and note that $z T(z)$ is anticausal. Thus, we can solve a Nehari problem with $z T(z)$ to obtain $L'(z)$, and recover $L(z)$ with $L(z) = z^{-1}L'(z)$. Next, we present the solution to our Nehari problem. We apply this idea to the anticausal part of $\Delta(z) K_0(z)$ denoted by $T(z)$ (given explicitly below).
\begin{lemma}\label{lemma:our_nehari}
The optimal solution to the Nehari problem with the anticausal transfer function $T(z)$ (given in \eqref{eq:TS})
% $$\{\Delta(z) K_0(z)\}_- = (R + B_u^\ast PB_u)^{-\ast/2}B_u^\ast (z^{-1} I - A_K^\ast )^{-1} A_K^\ast P B_w$$
is
\begin{align}\label{eq:sol_Nehari}
L(z) &= -(R + B_u^\ast PB_u)^{-\ast/2}B_u^\ast \Pi (zI- F_\gamma)^{-1}K_\gamma
\end{align}
where
\begin{align}\label{eq:th_Nehari_gain}
    K_\gamma &=  (I - A_K Z_\gamma A_K^\ast \Pi )^{-1}A_K Z_\gamma A_K^\ast P B_w\nn\\
    F_\gamma &= A_K - K_\gamma B_w^\ast P A_K,
        % & \rr{= A_K(I - Z_\gamma A_K^\ast \Pi A_K)^{-1}( I -   Z_\gamma \Pi )},
\end{align}
$\Pi$ is given in \eqref{eq:SC_regret_Lyapunov} and $Z_\gamma$ is given in \eqref{eq:Z_gamma}.
\end{lemma}

\subsection{The regret-optimal control problem (Theorem \ref{th:stricrly_regret})}\label{subsec:proof_ss}
By Theorem \ref{th:reg_as_Nehari}, we need explicit expressions for the factorization $\Delta^\ast(z^{-\ast})\Delta(z) = I+F^\ast(z^{-\ast}) F(z)$ and the decomposition of $\Delta(z)K_0(z)$ into its strictly-causal and anticausal transfer functions. These results are presented next as lemmas, and their proofs appear in Appendix \ref{app:technical_lemmas}.%\bb{\cite[App. \ref{app:technical_lemmas}]{extended}.} 
\begin{lemma}\label{lemma:fg}
For the state-space setting, the transfer functions of the operators $F$ and $G$ in \eqref{eq:operator_sys} are given by
\begin{equation*}
    F(z) \!=\! Q^{1/2}(z I \!-\! A)^{-1}B_uR^{-1/2}, ~~
    G(z) \!=\! Q^{1/2}(z I \!-\! A)^{-1}B_w
\end{equation*}
with $R = R^{\ast/2}R^{1/2}$ and $Q = Q^{\ast/2}Q^{1/2}$.
\end{lemma}

The canonical spectral factorization is presented next.
\begin{lemma}[Spectral factorization]\label{lemma:Delta}
The transfer function $I+ F^\ast(z^{-\ast})F(z)$ can be factored as $\Delta^*(z^{-*})\Delta(z)$, where
\begin{align}\label{eq:delta_def}
\!\!\!  \Delta(z) \!=\! (R + B_u^{*}PB_u)^{1/2} (I \!+\! K_{\emph{lqr}} (zI \!-\! A)^{-1}B_u)R^{-1/2}\!\!\!\!
\end{align}
$P$ is the unique stabilizing solution to the Ricatti equation
\begin{align}%\label{eq:Ricc_LQR}
Q - P + A^{*}PA - A^{*}PB_u(R + B_u^{*}PB_u)^{-1}B_u^{*}PA = 0,\nn
\end{align}
and $K_{\emph{lqr}} = (R + B_u^\ast P B_u)^{-1}B_u^\ast P A$. Furthermore, $\Delta^{-1}(z)$ is casual and bounded on the unit circle.
\end{lemma}

The following lemma provides the decomposition of the transfer function $\Delta(z)K_0(z)$.
\begin{lemma}[Decomposition]\label{lemma:decomposition}
The transfer function $\Delta(z)K_0(z)= - \Delta^{-*}(z^{-*})F^*(z^{-*})G(z)$ can be written as a sum of anticausal and strictly causal transfer functions
\begin{align}\label{eq:TS}
    T(z)&= -(R + B_u^\ast PB_u)^{-\ast/2}B_u^\ast (I+(z^{-1} I - A_K^\ast )^{-1} A_K^\ast) P B_w\nn\\
    S(z)&= -(R + B_u^\ast PB_u)^{-\ast/2}B_u^\ast P A (zI - A)^{-1}B_w,
\end{align}
\end{lemma}
Recall that Lemma \ref{lemma:our_nehari} (Sec. \ref{sec:main}) provides the solution to the Nehari problem for our problem.
\begin{proof}[Proof of Lemma \ref{lemma:our_nehari}]
We apply Theorem \ref{th:Nehari_general} with 
\begin{align}\label{eq:zTbar}
zT(z)&= \mspace{-5mu}-(R + B_u^\ast PB_u)^{-\ast/2}B_u^\ast (z^{-1} I - A_K^\ast )^{-1} P B_w
\end{align}
from Lemma \ref{lemma:decomposition} to obtain 
\begin{align}
 L'(z)&= -(R + B_u^\ast PB_u)^{-\ast/2}B_u^\ast \Pi (I+F_\gamma(zI- F_\gamma)^{-1})K_\gamma.\nn
\end{align}
Note that $zT(z)$ is bounded on the unit circle since $(A,B_u)$ is stabilizable so that the singular values of $A_K$ are strictly smaller than $1$. Finally, we compute $L(z) = z^{-1}L'(z)$.
\end{proof}
Using Lemmas \ref{lemma:Delta}-\ref{lemma:decomposition}, we can prove our main results.
\begin{proof}[Proof of Theorem \ref{th:stricrly_regret}]
To compute $\regret$, we apply Theorem \ref{th:Nehari_general} with $zT(z)$ in \eqref{eq:zTbar}. Recall that the optimal controller in \eqref{eq:th_reduction_con} is $K(z)= \Delta^{-1}(z)(L(z) + S(z))$. By Lemma \ref{lemma:Delta}, we have
\begin{align}
\Delta^{-1}(z)&= R^{1/2}(I - K_{\text{lqr}} (z I - A_K)^{-1}B_u) (R + B_u^\ast PB_u)^{-1/2}.\nn
\end{align}
The proof follows by computing the products in $K(z)$ and showing that one of the hidden states of the controller is equal to the state $x_t$. An extended proof is given in Appendix \ref{app:technical_lemmas}.%\bb{\cite[Th. ?]{extended}}.
% and \ref{lemma:decomposition}, we have
%     L(z) &= -(R + B_u^\ast PB_u)^{-\ast/2}B_u^\ast \Pi (zI- F_\gamma)^{-1}K_\gamma\nn\\
\end{proof}

\section{Conclusions}\label{sec:conclusion}
A novel controller is derived based on a regret criterion when compared to a clairvoyant controller with non-causal access to the entire disturbance sequence. The main difference from the classical $\mathcal H_\infty$ is its robustness against a clairvoyant controller rather than the classical robustness without a reference controller. The implementation of the regret-optimal controller is simple and is published in a public Git repository \cite{sabaggithubFI}. As illustrated in the numerical examples, the regret is a viable criterion and its potential should be assessed for other control systems. In two subsequent works, regret-based systems design has been utilized for filtering in \cite{SabagFilteringAISTATS} and the finite-horizon control problem studied in this paper \cite{gautam_FI_TV}.

\bibliography{ref}
\bibliographystyle{IEEEtran}
\clearpage
\appendices

\section{The non-causal controller (Theorem \ref{th:noncausal})}\label{app:non-causal}
The derivation of the non-causal controller is shown next.
\begin{proof}[Proof of Theorem \ref{th:noncausal}]
For any $w$, the non-casual sequence of control actions $v$ is the solution of $$\min_v  \|Fv + Gw\|_2^2 + \|v\|_2^2.$$ By a standard completion of the square, the objective can be written as
\begin{align}\label{eq:proof_reg_quadratic}
  &\|Fv + Gw\|_2^2 + \|v\|_2^2= \nn\\
  &(v\mspace{-2mu} + \mspace{-2mu}(I \mspace{-2mu}+\mspace{-2mu} F^{*}F)^{-1}F^*Gw)^* (I\mspace{-2mu}+\mspace{-2mu}F^*F) (v \mspace{-2mu}+\mspace{-2mu} (I \mspace{-2mu}+\mspace{-2mu} F^{*}F)^{-1} F^*Gw) \nn\\
  &\ + w^*G^*(I + FF^*)^{-1}Gw,
\end{align}
where $(I + F^{*}F)$ is invertible since it is a positive-definite operator and we also use $I-F(I+F^\ast F)^{-1}F^\ast = (I+ FF^\ast)^{-1}$. Note that $v$ is not assumed \textit{a priori} to be a linear function of $w$. Since the first term of \eqref{eq:proof_reg_quadratic} is non-negative, and the second term is independent of $v$, it is clear that that the linear mapping $v = K_0w$ with $K_0 = - (I + F^{*}F)^{-1}F^{*}G$ minimizes the cost.
\end{proof}

\section{Additional Numerical simulations}

\rr{We evaluate the Boeing 747 system with an integrator (i.e., an auto-regressive Gaussian process) disturbance:
$$w_t = n_t + \beta w_{t-1}, \ \ \ \ n_t\stackrel{\text{i.i.d.}}{\sim} \mathcal N(0,I).$$
The AR noise is parameterized with scalar $\beta$ that governs its spectral behavior. Similar to prior time-domain experiments, we run 30 independent trials and present the mean average cost as a function of time for different values of $\beta$ in Fig. \ref{fig:time_domain_colored}. For low $\beta$, the $\mathcal H_2$ controller attains the best performance. This is since the AR noise spectrum \emph{tends to} a white noise, but still it can be seen that the regret-optimal controller attains a control cost which is very close to the performance of the $\mathcal H_2$ controller. The other extreme, when $\beta$ is close to $1$, implies that the $\mathcal H_\infty$ controller outperforms the other controllers. Also here, the cost attained by the regret-optimal controller is close to the performance of the best controller. We also present the mid region where it can be seen that regret-optimal controller attains the best control cost among the other controllers. Overall, our time-domain evaluation shows that the regret-optimal controller either attains the best performance among the evaluated controllers or attains a performance which is relatively close to the lower cost that can be attained among these controllers. We thus say that regret-optimal controller attains the best of both worlds since it nicely interpolates between the $\mathcal H_2$ and the $\mathcal H_\infty$ controllers.}

\begin{figure}
    \centering
    \includegraphics[width=0.49\textwidth]{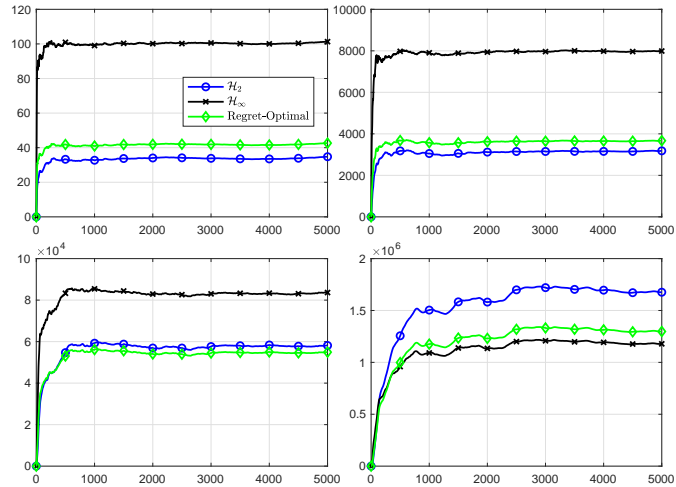}
    \caption{\rr{The control cost of the Boeing 747 example under an autoregressive noise parameterized with (a) $\beta \!=\! 0.1$, (b) $\beta \!=\! 0.5$, (c) $\beta \!=\! 0.9$, (d) $\beta \!=\! 0.99$. For small $\beta$, $\mathcal{H}_2$ outperforms other strategies and for large $\beta$, e.g. $\beta = 0.9$, regret-optimal control achieves the best performance. In the extreme case of $\beta$ close to $1$, $\mathcal{H}_{\infty}$ attains the superior performance.}}
    \label{fig:time_domain_colored}
\end{figure}

\section{Solution to the general Nehari problem}\label{app:Nehari_general}
\begin{proof}[Proof of Theorem \ref{th:Nehari_general}]
By Theorem $12.8.2$ in \cite{hassibi1999indefinite}, the optimal value of a Nehari problem is the maximal singular value of the Hankel operator of $T(z) = H(zI-F^\ast)^{-1}G$. The maximal singular value of the squared Hankel operator can be computed as the maximal eigenvalue of the product $\Pi Z$, where $Z\succeq0$ and $\Pi\succeq0$ are the controllability and observability Gramians, respectively. The Grammians can be computed as the solutions to the Lyapunov equations
\begin{align}
\Pi &= F \Pi F^\ast + G G^\ast\nn\\
 Z &= F^\ast Z F + H^\ast H.
\end{align}

The second part of Theorem \ref{th:Nehari_general} is the characterization of a solution that achieves a norm $\gamma$. The optimal solution will be derived as the \emph{central solution} of a more general solution that appears in \cite{hassibi1999indefinite}. In particular, we will utilize Lemma $12.8.1$ and Lemma $12.8.2$ in \cite{hassibi1999indefinite} as the starting point to simplify the central solution. Throughout the derivations, we use their notation and relate these to our notation at the proof's end.

The factorization 
\begin{align}\label{eq:proof_L_matrix}
    \begin{pmatrix} L_{11}(z)& L_{12}(z) \\ L_{21}(z) & L_{22}(z)
    \end{pmatrix}
    = R_e^{1/2} +
    \begin{pmatrix}
    -G^\ast\\ H\Pi F^\ast
    \end{pmatrix} (zI-F^\ast)^{-1}K_pR_e^{1/2},
\end{align}
appears in the proof of Lemma $12.8.1$, where the constants  
\begin{align*}
    K_p R_e^{1/2}&= \begin{pmatrix} -F^\ast PG & H^\ast + F^\ast P(I + GG^\ast P)^{-1}F\Pi H^\ast \end{pmatrix}\nn\\
    &\ \ \cdot \begin{pmatrix}
    (I+G^\ast P G)^{-\ast/2}&0\\
    0& \Delta^{-\ast/2}
    \end{pmatrix}\nn\\
    &= \begin{pmatrix} -F^\ast PG (I+G^\ast P G)^{-\ast/2} & \phi \end{pmatrix}\nn\\
    R_e^{1/2}& =
    \begin{pmatrix}
    (I+G^\ast PG)^{1/2} &0 \\
    -H\Pi F^\ast PG(I+G^\ast PG)^{-\ast/2}& \Delta^{1/2}
    \end{pmatrix}
\end{align*}
are taken from \cite[Eq. $12.8.26$]{hassibi1999indefinite}, $\Delta$ is defined as
\begin{align}
    \Delta&= \gamma^2I - H\Pi H^\ast -H\Pi F^\ast (P^{-1}+GG^\ast)^{-1} F\Pi H^\ast,
\end{align}
and $P$ is a solution to the Riccati equation $ P= F^\ast PF - K_pR_eK_p^\ast$. 

We now simplify the relevant coordinates in the second term of \eqref{eq:proof_L_matrix} for the central solution $-L_{21}(z)L^{-1}_{11}(z)$
\begin{align*}
    &\begin{pmatrix}
    -G^\ast\\ H\Pi F^\ast
    \end{pmatrix} (zI-F^\ast)^{-1}K_pR_e^{1/2}\nn\\
    &= \begin{pmatrix}
    G^\ast (zI-F^\ast)^{-1}F^\ast PG (I+G^\ast P G)^{-\ast/2} &\phi\\ -H\Pi F^\ast(zI-F^\ast)^{-1}F^\ast PG (I+G^\ast P G)^{-\ast/2}&\phi \end{pmatrix},
\end{align*}
so that $L_{11}(z)$ and $L_{21}(z)$ can be explicitly expressed as
\begin{align*}
    L_{11}(z)
    % &= (I+G^\ast PG)^{1/2} + G^\ast (zI-F^\ast)^{-1}F^\ast PG (I+G^\ast P G)^{-\ast/2} \nn\\
    &= [I + G^\ast (zI-F^\ast)^{-1}K_\gamma] (I+G^\ast PG)^{1/2}\nn\\
    L_{21}(z)&= -H\Pi F^\ast PG(I+G^\ast PG)^{-\ast/2} \nn\\
    &\ -H\Pi F^\ast(zI-F^\ast)^{-1}F^\ast PG (I+G^\ast P G)^{-\ast/2}\nn\\
    &= -H\Pi(I + F^\ast(zI-F^\ast)^{-1}F^\ast) F^\ast PG(I+G^\ast PG)^{-\ast/2},
\end{align*}
where we denoted $K_\gamma \triangleq F^\ast PG (I+G^\ast P G)^{-1}$ (the subscript $\gamma$ is due to $P$ that implicitly depends on $\gamma$).
% By the matrix inversion lemma, $L_{11}(z)$ can be computed as
% \begin{align*}
%     L_{11}^{-1}(z)&= (I+G^\ast PG)^{-1/2}[I + G^\ast (zI-F^\ast)^{-1} K_p]^{-1},
%     % &= (I+G^\ast PG)^{-1/2} (I - G^\ast(zI-F^\ast +K_p G^\ast  )^{-1}K_p )
% \end{align*}

The central solution can be explicitly written as
\begin{align}\label{eq:app_controller}
    &-L_{21}(z)L^{-1}_{11}(z)\nn\\
    % &= H\Pi(I + F^\ast(zI-F^\ast)^{-1}F^\ast) F^\ast PG(I+G^\ast PG)^{-1} [I + G^\ast (zI-F^\ast)^{-1}K_p]^{-1} \nn\\
   &= H \Pi( I + F^\ast(zI-F^\ast)^{-1} )K_\gamma(I+G^\ast (zI-F^\ast)^{-1} K_\gamma)^{-1}\nn\\
    &= H \Pi( I + F^\ast(zI-F^\ast)^{-1} )(I+K_\gamma G^\ast (zI-F^\ast)^{-1} )^{-1}K_\gamma\nn\\
    &= H \Pi z(zI-F^\ast)^{-1} (I+K_\gamma G^\ast (zI-F^\ast)^{-1} )^{-1}K_\gamma \nn\\
    &\stackrel{(a)}= H \Pi z (zI-F_\gamma)^{-1}K_\gamma\nn\\
    &= H \Pi (I + F_\gamma(zI- F_\gamma)^{-1} )K_\gamma,
\end{align}
where $(a)$ follows from $F_\gamma \triangleq F^\ast - K_\gamma G^\ast$.

Finally, let $Z_\gamma$ be the unique solution to the Lyapunov equation
\begin{align*}
 Z_\gamma &= F^\ast Z_\gamma F + \gamma^{-2}H^\ast H.
\end{align*}
Then, by Lemma $12.8.2$, the solution to the Riccati equation above is given $P = (I-Z_\gamma\Pi)^{-1}Z_\gamma$. This solution is utilized to simplify $K_\gamma$ as
\begin{align}
    K_\gamma &= F^\ast (I-Z_\gamma\Pi)^{-1}Z_\gamma G (I+G^\ast (I-Z_\gamma\Pi)^{-1}Z_\gamma G)^{-1}\nn\\
    % &= F^\ast (I-Z_\gamma\Pi)^{-1} (I+Z_\gamma GG^\ast (I-Z_\gamma\Pi)^{-1})^{-1}Z_\gamma G\nn\\
    &= F^\ast ( I-Z_\gamma\Pi + Z_\gamma GG^\ast )^{-1}Z_\gamma G\nn\\
    % &= F^\ast ( I-Z_\gamma\Pi + Z_\gamma (\Pi - F \Pi F^\ast) )^{-1}Z_\gamma G\nn\\
    &\stackrel{(a)}= F^\ast ( I - Z_\gamma F \Pi F^\ast )^{-1}Z_\gamma G,
\end{align}
where $(a)$ follows from $\Pi = F \Pi F^\ast + G G^\ast$.

To conclude the proof with our notation, we denote $L(z)$ in Theorem \ref{th:Nehari_general} as the central solution $-L_{21}(z)L^{-1}_{11}(z)$ in \eqref{eq:app_controller}.

% Recall that the regret-optimal controller can be recovered from the above solution as
% \begin{align*}
%   K(z) = \Delta^{-1}(z)(\bar{K}(z) + S(z)),
% \end{align*}
% where $ \bar{K}(z) = \bar{H}(zI - \bar{F})^{-1}\bar{G} + \bar{J}$.
% We showed in the main document that $\widehat{T}(z)$ can be written as
% \begin{align}
% \widehat{T}(z)&= H \Pi K_p - H\Pi K_pG^\ast(zI-F^\ast +K_p G^\ast  )^{-1}K_p \nn\\
%  &\ \ \ + H\Pi F^\ast(zI-F^\ast)^{-1} K_p - H\Pi F^\ast(zI-F^\ast)^{-1} K_p G^\ast(zI-F^\ast +K_p G^\ast  )^{-1}K_p,
% \end{align}
% where $K_p = F^\ast (I-Z_\gamma\Pi)^{-1}Z_\gamma G (I+G^\ast (I-Z_\gamma\Pi)^{-1}Z_\gamma G)^{-1}$.

% To simplify the controller, consider
% \begin{align}
%     &H \Pi K_p[I - G^\ast(zI-F^\ast +K_p G^\ast  )^{-1}K_p] + H\Pi F^\ast(zI-F^\ast)^{-1} K_p[I - G^\ast(zI-F^\ast +K_p G^\ast  )^{-1}K_p]\nn\\
%     % &= (H \Pi K_p + H\Pi F^\ast(zI-F^\ast)^{-1} K_p)(I - G^\ast(zI-F^\ast +K_p G^\ast  )^{-1}K_p) \nn\\
%     &\stackrel{(a)}= H \Pi( I + F^\ast(zI-F^\ast)^{-1} )K_p(I+G^\ast (zI-F^\ast)^{-1} K_p)^{-1}\nn\\
%     % &= H \Pi( I + F^\ast(zI-F^\ast)^{-1} )(I+K_pG^\ast (zI-F^\ast)^{-1} )^{-1}K_p\nn\\
%     % &= H \Pi z(zI-F^\ast)^{-1} (I+K_pG^\ast (zI-F^\ast)^{-1} )^{-1}K_p\nn\\
%     &\stackrel{(b)}= H \Pi z (zI-F_c )^{-1}K_p\nn\\
%     &= H \Pi (I + F_c(zI- F_c)^{-1} )K_p,
% \end{align}
% where $(a)$ follows from the Matrix inversion lemma and $(b)$ follows from $F_c = F^\ast - K_pG^\ast$.
\end{proof}

\section{Proofs of Technical lemmas $2-4$ and extended proof of Theorem $4$}\label{app:technical_lemmas}

\begin{proof}[Proof of Lemma \ref{lemma:fg}]
The linear operator $\mathcal T_F:\v{v}\to\v{s}$ can be represented as the state-space model
\begin{align}
    x_{t+1}&= Ax_t + B_u R^{-1/2}v_t\nn\\
    s_t &= Q^{1/2}x_t.
\end{align}
By taking the $z$-transform, we obtain:
\begin{align}
    z X(z)&= AX(z) + B_u R^{-1/2}V(z) \nn\\
    S(z) &= Q^{1/2}X(z),
\end{align}
so that $F(z) = Q^{1/2}(z I - A)^{-1}B_uR^{-1/2}$. The transfer function $G(z) = Q^{1/2}(z I - A)^{-1}B_w$ can be obtained similarly from the state-space model
\begin{align*}
    x_{t+1}&= Ax_t + B_w w_t \\
    s_t &= Q^{1/2}x_t.
\end{align*}
\end{proof}

\begin{proof}[Proof of Lemma \ref{lemma:Delta}]
By Lemma \ref{lemma:fg}, we have
\begin{align} \label{center-z-transform}
& I+ F^\ast(z^{-\ast})F(z) \\
&= I + R^{-\ast/2}B_u^\ast (z^{-1} I - A^\ast)^{-1}Q(z I - A)^{-1}B_uR^{-1/2}.\nn  
\end{align}
For ease of derivation, we will factor the term $R^{\ast/2}(I+ F^\ast(z^{-\ast})F(z))R^{1/2}$ as $\tilde{\Delta}^\ast(z^{-\ast})\tilde{\Delta}(z)$, and then the required factorization can be recovered as $\Delta(z) = \tilde{\Delta}(z)R^{-1/2}$.

First, we write $R^{\ast/2}(I+ F^\ast(z^{-\ast})F(z))R^{1/2}$ in the matrix form as
\begin{align}\label{eq:proof_factorization}
&\begin{pmatrix} B_u^{*}(z^{-1}I - A^\ast)^{-1} & I \end{pmatrix} \begin{pmatrix} Q & 0 \\ 0 & R\end{pmatrix} \begin{pmatrix} (zI - A)^{-1}B_u \\ I \end{pmatrix}\nn\\
&\ =\begin{pmatrix} B_u^{*}(z^{-1}I - A^\ast)^{-1} & I \end{pmatrix} \begin{pmatrix} Q -P + A^{*}PA & A^{*}PB_u \\ B_u^{*}PA & R + B_u^{*}PB_u \end{pmatrix}\nn\\
&\ \cdot \begin{pmatrix} (zI - A)^{-1}B_u \\ I \end{pmatrix},
\end{align}
where the equality can be verified directly and holds for any Hermitian matrix $P$.

%   $$\begin{pmatrix} B^{*}(z^{-1}I - A^\ast)^{-1} & I \end{pmatrix} \begin{pmatrix} -P + A^{*}PA & A^{*}PB \\ B^{*}PA & B^{*}PB \end{pmatrix} \begin{pmatrix} (zI - A)^{-1}B \\ I \end{pmatrix} = 0. $$
% Therefore the right-hand side of equation (\ref{center-z-transform}) can be expressed as

The middle matrix in \eqref{eq:proof_factorization} can be factored as $$\begin{pmatrix} I & \Psi^{*}(P) \\ 0 & I \end{pmatrix} \begin{pmatrix} \Gamma(P) & 0   \\0 & R + B_u^{*}PB_u  \end{pmatrix} \begin{pmatrix} I & 0 \\ \Psi(P) & I \end{pmatrix}, $$ where $$\Gamma(P) \triangleq Q - P + A^{*}PA - A^{*}PB_u(R + B_u^{*}PB_u)^{-1}B_u^{*}PA$$ and $$\Psi(P) \triangleq  (R + B_u^{*}PB_u)^{-1}B_u^{*}PA.$$
Since $(A, B_u)$ is a stabilizable pair; then the Riccati equation $\Gamma(P) = 0$ has a unique Hermitian solution. Suppose $P$ is chosen to be this solution and define $K_{\text{lqr}} = \Psi(P)$. Finally, by defining
$$\tilde{\Delta}(z) =  (R + B_u^{*}PB_u)^{1/2} (I + K_{\text{lqr}} (zI - A)^{-1}B_u),$$
we obtain the desired factorization $$ R^{\ast/2}(I+ F^\ast(z^{-\ast})F(z))R^{1/2} = \tilde{\Delta}^\ast(z^{-\ast})\tilde{\Delta}(z).$$ Recall that $\Delta(z) = \tilde{\Delta}(z)R^{-1/2}$, so $$\Delta(z) =  (R + B_u^{*}PB_u)^{1/2} (I + K_{\text{lqr}} (zI - A)^{-1}B_u)R^{-1/2}.$$

Finally, it remains to check that this choice of $\Delta(z)$ is causal, and its inverse is causal and bounded on the unit circle.

To see that the inverse is bounded, by the Matrix Inversion Lemma, the poles are at the eigenvalues of the matrix $A - B_uK_{\text{lqr}}$. It is a stable since $P$ was chosen to be the unique Hermitian solution to the Ricatti equation, and hence its spectral radius is less than $1$, which due to the causality of $\Delta^{-1}(z)$ guarantees the boundedness of $\Delta^{-1}(z)$ on the unit circle.
\end{proof}

\begin{proof}[Proof of Lemma \ref{lemma:decomposition}]
Recall that we decompose the product $\Delta(z)K_0(z) = -\Delta^{-*}(z^{-*})F^*(z^{-*})G(z)$ with
\begin{align}\label{eq:proof_ALLoperators}
    \Delta^{-*}(z^{-*}) &= (R + B_u^\ast PB_u)^{-\ast/2}\nn\\
    & \ \ \cdot (I + B_u^\ast(z^{-1}I-A^\ast)^{-1}K_{lqr}^\ast)^{-1}R^{\ast/2}\nn\\
    F^*(z^{-*}) &= R^{-\ast/2}B_u^\ast (z^{-1} I - A^\ast)^{-1}Q^{\ast/2}\nn\\
    G(z) &= Q^{1/2}(zI - A)^{-1}B_w.
\end{align}

% First, that the middle part of $\Delta^{-*}(z^{-*})$ can be written by the matrix inversion lemma as
% \begin{align}
%     I - B^*(z^{-1}I - A_K^*)^{-1}K_{\text{lqr}}^* &= (I + B^\ast(z^{-1}I-A^\ast)^{-1}K_{lqr}^\ast)^{-1}
% \end{align}

First, consider the $\Delta^{-*}(z^{-*})F^*(z^{-*})$ (omitting constants on the sides)
\begin{align}\label{eq:proof_decomposition_causal_product}
 &(I + B_u^\ast(z^{-1}I-A^\ast)^{-1}K_{lqr}^\ast)^{-1} B_u^\ast (z^{-1} I - A^\ast)^{-1}\nn\\
 &= B_u^\ast (I + (z^{-1}I-A^\ast)^{-1}K_{lqr}^\ast B_u^\ast)^{-1}  (z^{-1} I - A^\ast)^{-1}\nn\\
 &= B_u^\ast (z^{-1} I - A^\ast + K_{lqr}^\ast B_u^\ast)^{-1}\nn\\
&=  B_u^\ast (z^{-1} I - A_K^\ast )^{-1}
\end{align}

We now multiply $G(z)$ with \eqref{eq:proof_decomposition_causal_product} and apply a decomposition as appear in \cite[Lemma $12.3.3$]{hassibi1999indefinite},
\begin{align}\label{eq:prood_decom_lasteq}
    & - \Delta^{-*}(z^{-*})F^*(z^{-*})G(z)\nn\\
    &= -(R + B_u^\ast PB_u)^{-\ast/2}B_u^\ast \nn\\
    &\cdot (z^{-1} I - A_K^\ast )^{-1}Q (zI - A)^{-1} B_w\nn\\
    &= - (R + B_u^\ast PB_u)^{-\ast/2}B_u^\ast \nn\\
    &\cdot [(z^{-1} I - A_K^\ast )^{-1}A_K^\ast W + W A (zI - A)^{-1}+ W]B_w,
\end{align}
where $W$ solves $Q - W + A_K^\ast WA = 0$ . Finally, note that $W=P$ solves the Lyapunov equation. We now identify the strictly causal part of \eqref{eq:prood_decom_lasteq} as $S(z)$ and the remaining terms as $T(z)$.
\end{proof}

\begin{proof}[Extended proof of Theorem \ref{th:stricrly_regret}]
To compute $\regret$, we apply Theorem \ref{th:Nehari_general} with $zT(z)$ in \eqref{eq:zTbar}. Recall that the optimal controller in \eqref{eq:th_reduction_con} is $K(z)= \Delta^{-1}(z)(L(z) + S(z))$. By Lemmas \ref{lemma:our_nehari}, \ref{lemma:Delta}, and \ref{lemma:decomposition}, we have
\begin{align}
    L(z) &= -(R + B_u^\ast PB_u)^{-\ast/2}B_u^\ast \Pi (zI- F_\gamma)^{-1}K_\gamma\nn\\
\Delta^{-1}(z)&= R^{1/2}(I - K_{\text{lqr}} (z I - A_K)^{-1}B_u) (R + B_u^\ast PB_u)^{-1/2},\nn
\end{align}
and ${S}(z)$, respectively. The proof follows by computing the products in $K(z)$ and showing that one of the hidden states of the controller is equal to the system state $x_t$. Each of the produces can be simplified as follows
\begin{align*}
    &\Delta^{-1}(z){S}(z)= -R^{1/2}K_{lqr}(zI-A_K)^{-1}B_w\nn\\
    &\Delta^{-1}(z){L}(z) = -R^{1/2}(I - K_{\text{lqr}} (z I - A_K)^{-1}B_u) \nn\\
    & \hspace{10mm}\cdot (R + B_u^\ast PB_u)^{-1}B_u^\ast {\Pi} (zI- F_\gamma)^{-1} K_\gamma,
\end{align*}
and can be organized as 
\begin{align}\label{eq:proof_z_controller}
    K(z)&=H(zI-F)^{-1}G,
\end{align}
with
\begin{align}
    H &= - R^{1/2} \begin{pmatrix} (R + B_u^\ast PB_u)^{-1}B_u^\ast \Pi  F_\gamma & K_{\text{lqr}}\end{pmatrix}\nn\\
    F&= \begin{pmatrix} F_\gamma &0\\  - B_u (R + B_u^\ast PB_u)^{-1}B_u^\ast \Pi &A_K \end{pmatrix}\\
    G&= \begin{pmatrix} K_\gamma \\ B_w - B_u(R + B_u^\ast PB_u)^{-1}B_u^\ast (PB_w + \Pi K_\gamma) \end{pmatrix}.\nn
\end{align}
We proceed to show that one of the hidden states of the controller is the system state $x_t$. Let $\xi^1,\xi^2$ be the hidden states of the controller in \eqref{eq:proof_z_controller}. Then, the controller can be written as
\begin{align}\label{eq:output}
    \begin{pmatrix}\xi^1_{t+1}\\\xi^2_{t+1}\end{pmatrix}&=
    F    \begin{pmatrix}\xi^1_t\\\xi^2_t\end{pmatrix} + G w_t\nn\\
    R^{1/2}u_t&= H     \begin{pmatrix}\xi^1_{t+1}\\\xi^2_{t+1}\end{pmatrix} 
    \end{align}
The control signal can be explicitly written as
\begin{align}
    u_t&= - (R + B_u^\ast PB_u)^{-1}B_u^\ast \Pi  F_\gamma \xi^1_t - K_{\text{lqr}} \xi^2_t.
\end{align}
The evolution of $\xi^2_t$ is
\begin{align}
    \xi^2_{t+1}
    &= A \xi^2_t + B_u u_t + B_w w_t\nn\\
    &\stackrel{(b)}= x_{t+1},
\end{align}
where $(a)$ follows from the control signal in \eqref{eq:output} and $(b)$ follows from an inductive argument on the evolution of $\xi^2_{t}$.

Finally, the evolution of $\xi_{t}^1$ is given by 
\begin{align}
    \xi^1_{t+1}&= F_\gamma \xi^1_t + K_\gamma w_t\nn\\
    \tilde{u}_t&= (R + B_u^\ast PB_u)^{-\ast/2}B_u^\ast \Pi  F_\gamma \xi^1_t \nn\\
    & \ + (R + B_u^\ast PB_u)^{-\ast/2}B_u^\ast \Pi  K_\gamma w_t.
\end{align}
To simplify the presentation of the controller, we scale $\tilde{u_t}$ with $-(R + B_u^\ast PB_u)^{-1/2}$ as
$\hat{u}_t = -(R + B_u^\ast PB_u)^{-1/2}\tilde{u_t}$ and omit the redundant superscript. To conclude, the control signal can be written as
\begin{align*}
    u_t&= \hat{u}_t - K_{\text{lqr}} \xi^2_t.
\end{align*}
\end{proof}
\end{document}